\documentclass[10pt,reqno,a4paper]{amsart}
\usepackage{euscript}
\usepackage{amssymb}
\usepackage{breqn}
\usepackage{graphicx}
\usepackage{color}
\usepackage{subfigure,graphicx}
\usepackage{overpic}
\usepackage{float}
\usepackage{soul}
\usepackage{hyperref}
\graphicspath{ {./images/} }
\DeclareGraphicsExtensions{.png,.pdf}

\newcommand{\R}{\ensuremath{\mathbb{R}}}

\newcommand{\ov}{\overline}

\newcommand{\x}{\bx }
\newcommand{\y}{\boldsymbol{y}}

\newcommand{\bu}{\boldsymbol{u}}
\newcommand{\bv}{\boldsymbol{v}}
\newcommand{\q}{\mathbf{q}}
\newcommand{\f}{\boldsymbol{f}}
\newcommand{\g}{\boldsymbol{g}}

\newcommand{\w}{\boldsymbol{w}}

\newcommand{\B}{\mathcal{B}}
\newcommand{\A}{\mathcal{A}}
\newcommand{\parde}[2][]{\frac{\partial#1}{\partial#2}}

\newcommand{\e}{\varepsilon}

\newcommand{\C}{\ensuremath{\mathcal{C}}}

\newcommand{\bx}{\mathbf{x}}

\newcommand{\by}{\mathbf{y}}
\newcommand{\bz}{\mathbf{z}}
\newcommand{\bp}{\mathbf{p}}

\def\e{\varepsilon}

\newtheorem {theorem} {Theorem} 
\newtheorem {proposition} {Proposition}

\newtheorem {lemma} {Lemma}

\newtheorem {remark} {Remark}

\newtheorem {mtheorem} {Theorem}

\begin{document}
\allowdisplaybreaks
\title[Torus bifurcation in non-smooth systems]{Detecting Limit Tori in Non-Smooth Systems:\\ An Analytic Approach with Applications\\ to 3D Piecewise Linear Systems}

\begin{abstract}
This work investigates a class of non-autonomous $T$-periodic piecewise smooth differential systems and their associated time-$T$ maps.  Our main result provides an analytical approach for detecting, within this class of piecewise differential systems, isolated invariant tori associated with normally hyperbolic invariant closed curves of the time-$T$ map. To achieve this, we derive sufficient conditions under which smooth near-identity maps undergo a Neimark--Sacker bifurcation. As an application of our main result, we present a family of 3D piecewise linear differential systems exhibiting attracting and repelling isolated invariant tori which, moreover, persist under small perturbations. To the best of our knowledge, this family provides the first examples in which limit tori are analytically detected in piecewise linear systems.
\end{abstract}

\author{Murilo R. C\^{a}ndido$^1$, Douglas D. Novaes$^2$, and Joan S.G. Rivera$^2$}

\address{$^1$ Department of Mathematics and Computer Science, Faculty of Science and Technology São Paulo State University (UNESP), Rua Roberto Simonsen, 305 - Centro Educacional, P. Prudente, 19060-900, São Paulo, Brazil} 
\address{$^2$ Departamento de Matem\'{a}tica - Instituto de Matem\'{a}tica, Estat\'{i}stica e Computa\c{c}\~{a}o Cient\'{i}fica (IMECC) - Universidade
Estadual de Campinas (UNICAMP), \ Rua S\'{e}rgio Buarque de Holanda, 651, Cidade Universit\'{a}ria Zeferino Vaz, 13083-859, Campinas, SP,
Brazil}
\email{mr.candido@unesp.br}
\email{ddnovaes@unicamp.br}
\email{jsgaitanr@gmail.com}




\maketitle

\section{Introduction and outline of the main results}\label{section1}
Many natural phenomena display abrupt transitions in behavior, such as collisions, dry friction changes, or biological switches, that cannot be adequately described by smooth models. These discontinuities in the governing equations require mathematical frameworks capable of rigorously handling nonsmooth dynamics (see \cite{jeffrey18,jeffrey20}). To address the fundamental issue of defining solutions in such contexts, Filippov introduced a formulation based on differential inclusions, now known as Filippov’s convention, which provides a systematic approach for defining trajectories across discontinuity surfaces \cite{Filippov}. Systems governed by this framework, called Filippov systems, exhibit discontinuities along a switching set, typically assumed to be a smooth manifold. This formulation has found successful application in diverse fields, including sliding mode control \cite{utkin1978sliding,utkin2017sliding}, neuronal dynamics \cite{izhikevich2007dynamical}, ecological dynamics \cite{predatorprey,Kivan2011,Piltz2014}, mechanical and electromagnetic systems \cite{PiecewiseBernardo,jeffrey2018hidden,example4,example1}, and even market modeling \cite{MOLINADIAZ2024246}.

As in smooth dynamical systems, compact invariant sets serve as key structures for describing the dynamics. The first examples of such objects are equilibria and limit cycles. In higher-dimensional systems, invariant tori often arise, typically associated with periodic or quasiperiodic motion (see, for instance, \cite{GREBOGI1985354}). These geometric structures confine trajectories within bounded regions of the phase space, and their stability properties determine whether nearby trajectories converge toward or diverge from the torus.

The present study is structured around two primary goals. First, we develop an analytical method for deriving sufficient conditions that ensure the existence and stability of isolated invariant tori in a family of non-autonomous $T$-periodic piecewise smooth differential systems. These invariant tori are called {\it limit tori} and they are associated with invariant normally hyperbolic closed curves of the time-$T$ map, which, as we show, can be expressed as smooth near-identity maps. Second, we apply this theoretical framework to analyze a specific family of 3D piecewise linear differential systems, providing, to the best of our knowledge, the first analytical proof of the existence of limit  tori in such systems.

In the continuation of this section, we give a brief overview of our main results. Precise statements and rigorous formulations are postponed to the subsequent sections.

\subsection{Torus bifurcation in a class of nonsmooth system}

Averaging theory provides a classical and powerful framework for establishing the persistence of periodic motions in differential equations of the form $\dot{\bx}=\e F(t,\bx,\e)$. While its effectiveness in detecting periodic solutions is well known, considerable effort has been devoted to adapting this perturbative methodology to capture higher dimensional invariant objects, in particular invariant tori. Early contributions by Krylov, Bogoliubov, and Mitropolski\u{\i}~\cite{BM, BK} demonstrated that averaging techniques can be used to infer the presence of invariant tori in differential systems of kind $\dot{\bx}=\e F_1(t,\bx)$. Subsequent developments by Hale~\cite{hale1961integral} established a fundamental link between invariant tori of such a system and limit cycles of the autonomous differential system
\[
\dot \bx=\frac{1}{T}\int_0^T F_1(t,\bx)\,dt,
\]
which corresponds to the time average of the original one. More recently, Novaes and Pereira~\cite{novaes2024invariant} extended Hale's result by proving that this correspondence holds for a broader class of differential systems and at arbitrary orders of averaging. In the planar non-autonomous setting, these ideas were further refined by Novaes, Pereira, and C\^andido~\cite{pereira2023mechanism}, who provided conditions ensuring the bifurcation of normally hyperbolic invariant tori. Their approach is based on the continuation results of Chicone and Liu~\cite{chicone2000continuation} for rapidly oscillating systems. Finally, high-order averaging methods were also employed by Cândido and Novaes~\cite{candido2020torus} to show that a Hopf bifurcation in the truncated averaged equations gives rise to a torus bifurcation in the original system (see also~\cite{NP25}).

Although averaging theory has been extensively developed over the past decade to analyze periodic behavior of non-smooth differential systems (we can mention \cite{llibre2015averaging,LNM15,llibre2017averaging}), its application to the detection of isolated invariant tori remains largely unexplored. The present work constitutes a first step in this direction, laying the groundwork for extending averaging-based techniques for studying higher-dimensional limit sets in non-smooth differential systems and opening new avenues for future research.

Consider the following family of piecewise differential systems,
\begin{equation}\label{PSDE}
\dot{\bx }(t)=\sum_{i=1}^k \e^i F_i(t,\bx ;\alpha)+\e^{k+1} R(t,\bx ;\alpha,\e),\quad (t,\bx ;\alpha,\e)\in S^1\times \ov D\times \ov I \times (-\e_0,\e_0),
\end{equation}
where $S^1=\mathbb{R}/T\mathbb{Z},$ with $T>0$, $D\subset\mathbb{R}^2$ and $I\subset\mathbb{R}$ are open and bounded sets, and $\varepsilon_0>0$ is sufficiently small. In addition, $F_i$, $i=1,\ldots,k$, and $R$ are defined as
\[
F_i(t,\bx ;\alpha)=
\begin{cases}
F_i^0(t,\bx ;\alpha), & 0<t<\theta_1(\bx ;\alpha),\\[1mm]
F_i^1(t,\bx ;\alpha), & \theta_1(\bx ;\alpha)<t<\theta_2(\bx ;\alpha),\\[1mm]
\vdots\\[1mm]
F_i^n(t,\bx ;\alpha), & \theta_n(\bx ;\alpha)<t<T,
\end{cases}
\]
and
\[
R(t,\bx ;\alpha,\e)=
\begin{cases}
R^0(t,\bx ;\alpha,\e), & 0<t<\theta_1(\bx ;\alpha),\\[1mm]
R^1(t,\bx ;\alpha,\e), & \theta_1(\bx ;\alpha)<t<\theta_2(\bx ;\alpha),\\[1mm]
\vdots\\[1mm]
R^n(t,\bx ;\alpha,\e), & \theta_n(\bx ;\alpha)<t<T,
\end{cases}
\]
where the switching functions $\theta_j\colon \ov D\times \ov I\to S^1$, $j=1,\ldots,n$, as well as the functions $F_i^j:S^1\times \ov D\times \ov I \times \to\R^n$ and $R^j: S^1\times \ov D\times \ov I \times (-\e_0,\e_0)\to\R^n$, $i,j\in\{1,\ldots,n\}$, are assumed to be of class $C^r$, for some $r\ge k+1$.

We investigate the dynamics of system~\eqref{PSDE} in the extended phase space 
$\mathbb{S}^1 \times D$. To this end, we consider the associated time-$T$ map defined on $D$ as
\[
\mathcal{P}_T(\boldsymbol{x};\alpha,\varepsilon) := \varphi(T,\boldsymbol{x};\alpha,\varepsilon),
\]
where $\varphi(\cdot,\cdot;\alpha,\varepsilon)$ denotes the flow of~\eqref{PSDE}. 
By the Picard--Lindel\"of Theorem, applied on each region of the piecewise-defined 
differential system \eqref{PSDE}, one can find $\varepsilon_0>0$ small enough such that, 
for every $(\boldsymbol{x},\alpha,\e)\in \overline{D}\times\overline{I}\times(-\e_0,\e_0)$, 
$\varphi(t,\boldsymbol{x};\alpha,\varepsilon)$ is defined for all $t\in[0,T]$. 
Consequently, $\mathcal{P}_T(\cdot;\alpha,\e)$ is well defined for each $\alpha\in \overline{I}\times(-\e_0,\e_0)$.

Analogously to periodic solutions, which is associated to fixed points of $\mathcal{P}_T$, a closed curve $\Gamma\subset D$ satisfying $\mathcal{P}_T(\Gamma;\alpha,\e)=\Gamma$ implies the existence of an invariant torus in $\mathbb{S}^1 \times D$. Accordingly, throughout this work we define an invariant torus of system~\eqref{PSDE} as an invariant closed curve of its associated time-$T$ map.

Based on this characterization, Theorem~\ref{MainTheorem} establishes sufficient conditions ensuring that system~\eqref{PSDE} undergoes a torus bifurcation, which is equivalent to a Neimark--Sacker bifurcation in  the associated time-$T$ map. This latter bifurcation is characterized by the emergence of a normally hyperbolic invariant closed curve due to a change in the stability of a fixed point (see, for example, \cite{kuznetsov,NP25}). Theorem~\ref{MainTheorem} extends the results of~\cite{candido2020torus} to the broader setting of non-smooth systems.

The core of our approach is to prove that the time-$T$ map $\mathcal{P}_T$ depends smoothly on the variables $(\bx,\alpha,\varepsilon)$ and therefore admits a near-identity expansion, obtained via a Melnikov procedure, of the form
\[
\mathcal{P}_T(\bx;\alpha,\e)=\boldsymbol{f}(\bx;\alpha,\e)
= \bx + \e^r \boldsymbol{f}_r(\bx;\alpha) + \cdots + \e^{k} \boldsymbol{f}_{k}(\bx;\alpha)
+ \e^{k+1}\bar{F}(\bx,\alpha,\e),
\]
with $r\in\{1,\dots,k\}$. This allows us to reduce the problem to the analysis of a general smooth near-identity map. Following the framework developed in~\cite{candido2020torus}, Theorem~\ref{TheoremA} provides sufficient conditions under which such a map undergoes a Neimark--Sacker bifurcation. Accordingly, Theorem~\ref{MainTheorem} follows from Theorem~\ref{TheoremA}.

\subsection{Limit tori in 3D piecewise linear systems} 

In contrast to planar linear differential systems, which are known to be devoid of limit cycles, their planar piecewise counterparts may exhibit such phenomena. Early examples appeared in the book by Andronov \emph{et al.}~\cite{AndronovEtAl66} published in 1937. Since then, research has largely focused on planar piecewise linear systems with a single straight switching line, which nevertheless are capable of supporting rich dynamics. In this setting, Lum and Chua~\cite{lum1991global} constructed continuous systems with a single limit cycle and conjectured that continuity across the switching line prevents the appearance of more limit cycles. Freire \emph{et al.} proved this conjecture in~\cite{freire1998bifurcation} (see also \cite{CarmonaEtAl19b}). When the continuity assumption is relaxed, richer behaviors emerge. In this case, the first examples exhibiting two limit cycles were provided by Han and Zhang~\cite{han2010hopf} in 2010, who claimed that this number should be maximal. This conjecture was subsequently disproved by Huan and Yang~\cite{huan2012number}, who presented numerical evidence of examples with three limit cycles, which was later analytically confirmed by Llibre and Ponce~\cite{llibre2012three}. More recently, Carmona \emph{et al.}~\cite{Carmona2023} established the finiteness of the number of limit cycles for such systems.

Increasing the dimension, although linear 3D differential systems cannot support isolated invariant tori, in light of the preceding discussion, it is natural to ask whether their piecewise counterparts may admit isolated invariant tori. 

In this direction, we apply our results to prove in Theorem~\ref{mainresult} that the following $(\varepsilon,\alpha,b)$-family of piecewise linear differential systems exhibits a limit torus for suitable values of the parameters $\varepsilon,\alpha,b$. This limit torus is attracting for $b>0$ and repelling for $b<0$, and moreover persists under sufficiently small perturbations:
\begin{equation}\label{ds0intro}
	\begin{pmatrix}
x' \\
y' \\
z'
\end{pmatrix}
=
\begin{pmatrix}
-y \\
x \\
0
\end{pmatrix}
+\left\{
\begin{aligned}   
&\varepsilon\, A^{+}
\begin{pmatrix}
x \\
y \\
z
\end{pmatrix},& y>0,\\[0.2em]
&\varepsilon\, A^{-}
\begin{pmatrix}
x \\
y \\
z
\end{pmatrix}
+\varepsilon^2 B^{-}
\begin{pmatrix}
x \\
y \\
z
\end{pmatrix}, & y<0,
\end{aligned}
\right.
\end{equation}
where
\begin{equation}\label{ds0intro1}
A^{+}=\begin{pmatrix}
	0 & 0 & -\dfrac{4 \pi ^2 b}{8+9 \pi ^2} \\
	0 & -1 & -1 \\
	0 & 0 & \dfrac{1}{2}
\end{pmatrix},
\qquad
A^{-}=\begin{pmatrix}
	\alpha & 4 & -1 \\
	0 & 0 & 0 \\
	0 & -1 & 0
\end{pmatrix},
\qquad
B^{-}=\begin{pmatrix}
	0 & 0 & 0 \\
	0 & 0 & b \\
	0 & 0 & 0
\end{pmatrix}.
\end{equation}
To the best of our knowledge, this is the first 3D piecewise linear differential system for which the existence of a limit torus has been established analytically (see Figure~\ref{AppTorus} for its numerical detection).

\begin{figure}[H]
\begin{centering}
\begin{overpic}[width=12cm]{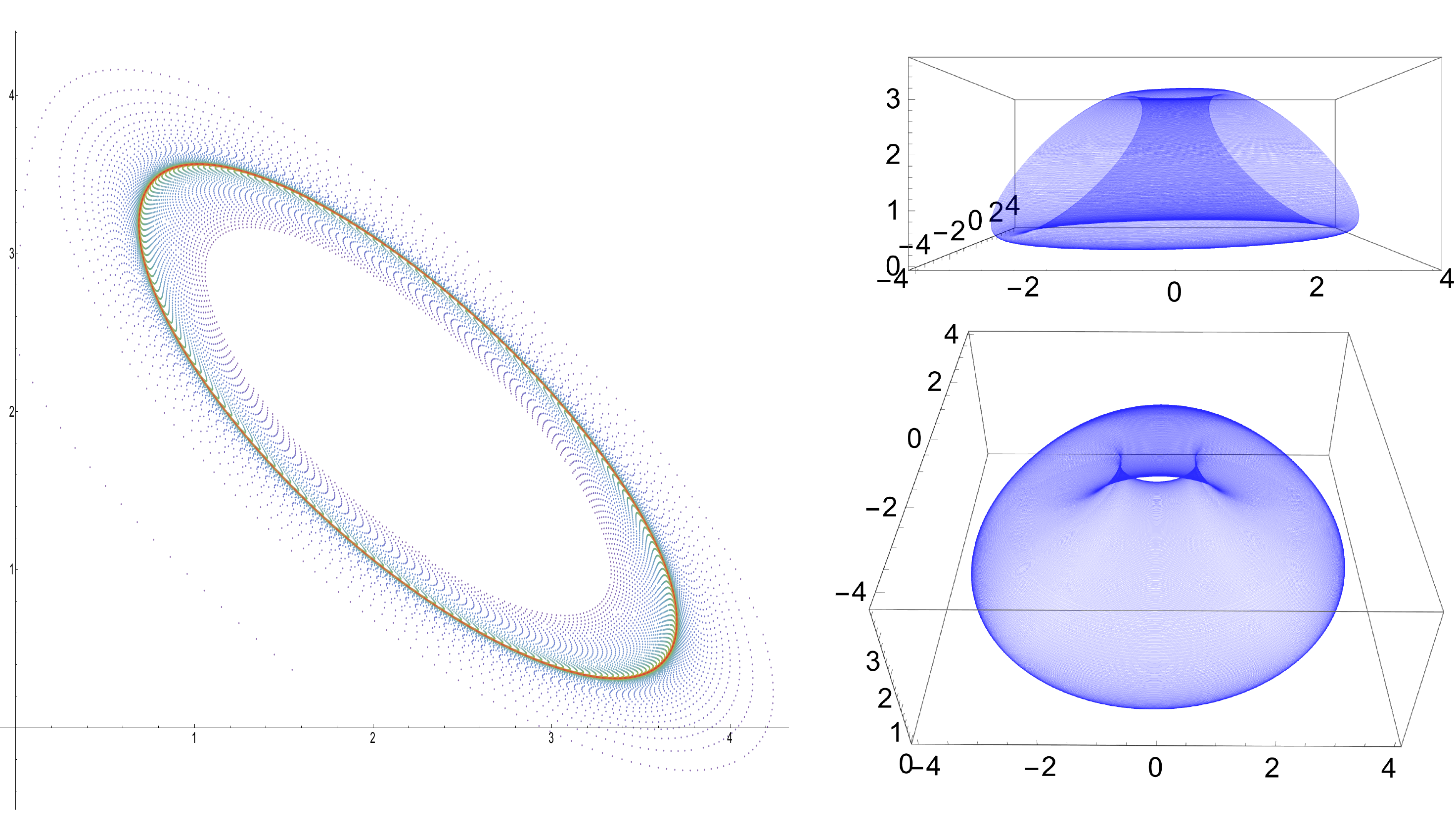}
\end{overpic}
\caption{%
Numerical simulation of systems~\eqref{ds0intro}--\eqref{ds0intro1} for $b=-5$, $\varepsilon=1/40$, and
$\alpha=\varepsilon(\pi^{2}/8-2)$. The left panel shows an invariant closed curve on the Poincaré section
$y=0$, $x>0$, of the Poincaré map, corresponding to an invariant torus. The right panel presents two views
of a trajectory in phase space, close to the associated invariant torus, obtained by numerically
integrating the system with initial condition $(3.669234340877,\,0,\,0.48488236396962971)$ over the time interval
$t\in[0,10000]$.}
\label{AppTorus}
\end{centering}
\end{figure}

\subsection{Organization of the paper}
Section~\ref{prel} introduces the basic notions and preliminary results required throughout the paper. In particular, Subsection~\ref{section2} revisits the perturbative methods for non-smooth systems, while Subsection~\ref{section3} reviews the Neimark--Sacker bifurcation for maps. In Section~\ref{section4}, we establish the smoothness properties of the time-$T$ map associated with piecewise smooth systems of the form~\eqref{PSDE}. Based on that, Section~\ref{section5} derives, in Theorem~\ref{TheoremA}, sufficient conditions for the occurrence of a Neimark--Sacker bifurcation in smooth near-identity maps. In Section~\ref{section6}, we extend in Theorem~\ref{MainTheorem} the results of~\cite{candido2020torus} to the class of piecewise smooth systems~\eqref{PSDE}, thereby obtaining sufficient conditions that guarantee the emergence of a limit torus in~\eqref{PSDE}. Finally, Section~\ref{AppliTorus} applies the developed theoretical framework to construct, in Theorem~\ref{mainresult}, a family of 3D piecewise linear differential systems that exhibits a limit torus.

\section{Basic notions and preliminaries results}\label{prel}
\subsection{Melnikov and Averaging methods}\label{section2}

Melnikov analysis and averaging theory are fundamental tools in dynamical systems and are widely used to detect isolated periodic solutions in systems of the form~\eqref{PSDE}. In the smooth setting, these two methods are closely related and yield equivalent bifurcation functions (see, for instance, \cite{novaes2021higher}), that is, functions whose simple zeros correspond to periodic solutions of the differential equation. In the non-smooth context, however, this equivalence does not always hold. In particular, several works have employed these techniques to study bifurcation phenomena and to estimate the number of limit cycles in differential systems of type~\eqref{PSDE}. For example, explicit formulas for bifurcation functions of arbitrary order were derived in \cite{llibre2017averaging} in the case where the functions $\theta_i$ in~\eqref{PSDE} are constant. In this setting, the first- and second-order bifurcation functions are obtained from~\cite[Proposition~2]{llibre2017averaging} and are given by
\begin{equation}\label{prom}
\begin{array}{l}
\displaystyle \g_1(\bx)=\int_{0}^{T}F_1(s,\bx)\,ds\quad \text{and}\vspace{0.3cm}\\
\displaystyle  \g_2(\bx)= \int_{0}^{T}\bigg[D_xF_1(s,\bx)
\int_{0}^{s}F_1(t,\bx)dt+F_2(s,\bx)\bigg]ds.
\end{array}
\end{equation}
These expressions coincide with the first- and second-order bifurcation functions in the smooth setting and are commonly referred to as averaged functions. However, as observed in~\cite{llibre2015averaging}, the function $\g_2$ does not always govern the bifurcation of isolated periodic solutions in non-smooth differential systems. In such cases, additional degeneracy conditions on the switching manifold are required to ensure that $\g_2$ effectively plays the role of the second-order bifurcation function. This issue is highlighted in~\cite[Theorem B]{llibre2015averaging}, where hypothesis (Hb2) is introduced precisely to guarantee that $\g_2$ serves as the second-order bifurcation function.

The derivation of higher-order bifurcation functions for the general differential system~\eqref{PSDE} requires obtaining a series expansion of the associated time-$T$ map around $\varepsilon = 0$, namely,
\begin{equation}\label{mapintroT}
\mathcal{P}_{T}(\bx;\alpha,\e) = \bx+\e\, \Delta_{1}(\bx,\alpha)+\e^{2}\Delta_{2}(\bx,\alpha)+\cdots+\e^{k}\, \Delta_{k}(\bx,\alpha)+\mathcal{O}(\e^{k+1}).
\end{equation}
The computation of this expansion is commonly referred to as the Melnikov procedure, and the function $\Delta_i$, for $i\in\{1,\ldots,k\}$, is called the Melnikov function of order~$i$.  As shown in~\cite{bastos2019melnikov},
\begin{equation}\label{mel1}
\Delta_1(\bx,\alpha)=\g_1(\bx,\alpha)\quad \text{and}\quad \Delta_2(\bx;\alpha)= \g_2(\bx;\alpha) + \widetilde{\g}_2(\bx;\alpha),
\end{equation}
where $\g_1$ and $\g_2$ are defined in \eqref{prom}, and
\[
\widetilde{\g}_2(\bx;\alpha)  = \sum_{j=1}^n \Big(F_1^{j-1}(\theta_j(\bx),\bx,\alpha) -
F_1^{j}(\theta_j(\bx),\bx,\alpha)\Big)
D_{\bx}\theta{_j}(\bx)\int_0^{\theta_{j}(\bx)}\hspace{-0.3cm}\,F_1(s,\bx,\alpha)\,ds.
\]

Notice that the second-order averaged function $\g_2(x)$ deviates from the second-order Melnikov function of \eqref{PSDE} due to the additional term $\widetilde{\g}_2(\mathbf{x};\alpha)$, which depends on the discontinuity jump and on the geometry of the switching manifold, quantified respectively by 
\[
F_1^{j-1}(\theta_j(\bx),\bx) -
F_1^{j}(\theta_j(\bx),\bx)\quad \text{and}\quad D_\bx\theta_j(\bx)\int_0^{\theta_j(\bx)}\hspace{-0.3cm}F_1(s,\bx)ds.
\] 

We emphasize that the applicability of bifurcation theory to the time-$T$ map~\eqref{mapintroT} does not depend on whether the underlying differential equation is smooth or piecewise smooth, but rather on the regularity properties of the associated time-$T$ map. Accordingly, in Section~\ref{section4} we prove that the time-$T$ map of~\eqref{PSDE} is smooth.

\subsection{Neimark-Sacker Bifurcation}\label{section3}

Consider the discrete dynamical system in the plane described by
\begin{equation}\label{two-map}
	\bx\mapsto \f(\bx,\alpha), 
\end{equation}
wherein $\f$ represents a map possessing sufficient smoothness, $\bx=(x_{1},x_{2})^{T}$, and $\alpha\in\R$. Furthermore, we assume that $\bx=0$ is a fixed point for $\f$ at $\alpha=0$, and that the Jacobian matrix $\mathbf{A}=D_{\bx}\, \f(0,\alpha)$ exhibits simple eigenvalues $\lambda_{1,2}=e^{\pm i\theta_{0}}$, with $0<\theta_{0}<\pi$. Given the smoothness properties of $\f$, the Taylor expansion of $\f(\bx,0)$ in the neighborhood of $\bx=0$ can be expressed as follows

\begin{equation*}\label{expT}
\f(\bx,0)=\mathbf{A}\,\bx +\frac{1}{2}\,\mathbf{B}(\bx,\bx) +\frac{1}{6}\,\mathbf{C}(\bx,\bx,\bx) + \mathcal{O}(||\bx||^4),    
\end{equation*}
where $\mathbf{A}=D_{\bx}F(0,0)$, and $B(\bx,\by)$ and $C(\bx,\by,\bz)$ are, respectively the bilinear and trilinear forms that appear in the Taylor expansion. 

Now, let be $\q\in\C^{2}$ a complex eigenvector of $\mathbf{A}$ related to the eigenvalue $e^{i\theta}$, i.e., $\mathbf{A}\,\q=e^{i\theta}\q$. Analogously, consider $\bp \in \C^{2}$ to be an eigenvector of $\mathbf{A}^{T}$ related to the eigenvalue $e^{-i\theta}$, which implies $\mathbf{A}^{T}\,\bp=e^{-i\theta}\bp$. In this context, $\langle\bp,\q\rangle=\bar{\bp}^{T}\q$ refers to the conventional inner product within $\C^{2}$. 

Assuming $e^{ik\theta_{0}}\neq 1$ for $k\in\{1,2,3,4\}$, the coefficient $c_{1}$ can be computed using the following formula
\begin{equation}\label{c1(0)formula}
c_{1}=  \frac{g_{20}\,g_{11}(1-2\lambda_{0})}{2(\lambda_{0}^{2}-\lambda_{0})}  + \frac{|g_{11}|^{2}}{1-\bar{\lambda}_{0}} + \frac{|g_{02}|^{2}}{2(\lambda_{0}^{2}-\bar{\lambda}_{0})} +\frac{g_{21}}{2}\cdot
\end{equation}
Here, $\lambda_{0}=e^{i\theta_{0}}$, and for $\alpha=0$ we have
\begin{equation}\label{expressionsg_ij's}
  g_{02}=\left<\mathbf p,B(\mathbf {\bar{q}},\mathbf {\bar{q}})\right>,\,
  g_{11}=\left<\mathbf p,B(\mathbf q,\mathbf {\bar{q}})\right>,\,
  g_{20}=\left<\mathbf p,B(\mathbf q,\mathbf q)\right>,\,   g_{21}=\left<\mathbf p, C(\mathbf q,\mathbf q, \mathbf {\bar{q}})\right>.   
 \end{equation}
 
 We refer to the scalar $\ell_{1}=\text{Re}[e^{-i\theta_{0}}\,c_{1}]$ as the \textit{first Lyapunov coefficient}. Moreover, from Equation \eqref{c1(0)formula}, we can get the following formula 
\begin{equation}\label{firstLyapunovcoefficient}
     {\ell_{1}}= \textrm{Re}\left(\frac{e^{-i\theta}g_{21}}{2} \right) - \textrm{Re}\left(\frac{(1-2\,e^{i\theta})\,e^{-2i\theta}}{2(1-e^{i\theta})}\,g_{20}g_{11}\right) -\frac{1}{2}|g_{11}|^{2} -\frac{1}{4}|g_{02}|^{2},
\end{equation}
where $g_{20}$, $g_{21}$, $g_{02}$ and $g_{11}$ can be computed using the equation \eqref{expressionsg_ij's}.

In what follows, we state the Neimark--Sacker bifurcation Theorem.

\begin{theorem}\cite[page 135]{kuznetsov}\label{NSBTheorem}
Assume that for $|\alpha|$ sufficiently small $\bx=0$ is a fixed point of the map \eqref{two-map} having eigenvalues $r(\alpha)\,e^{\pm i\varphi(\alpha)}$ with $r(0)=1$ and $\varphi(0)=\theta$, $0<\theta<\pi$. Suppose that the following assertions are satisfied
\begin{itemize}
    \item[(1).] $r'(0)\neq 0$,
    \item[(2).] $e^{ik\theta}\neq 1$, for $k=1,2,3,4$,
    \item[(3.)] $\ell_{1}\neq 0$.
\end{itemize}
Then, there exist an open set $U\subset \R^{2}$ around the fixed point $\bx=0$ and an open interval $I\subset \R$ around $\alpha=0$ such that 
\begin{itemize}
 \item When $\alpha\in I$ and $\ell_{1}<0$, we have the following possibilities: if $\alpha \,r'(0)\leq 0$, then $\x=0$ is attracting; if $\alpha \,r'(0)>0$, then $\x=0$ is repelling and there exists a unique asymptotically attracting closed invariant curve $\gamma_{\alpha}$ in $U$ of the map \eqref{two-map}.
    \item When $\alpha\in I$ and $\ell_{1}>0$, we have the following possibilities: if $\alpha \,r'(0)\geq 0$, then $\x=0$ is repelling; if $\alpha\, r'(0)<0$, then $\x=0$ is attracting and there exists a unique repelling closed invariant curve $\gamma_{\alpha}$ in $U$ of the map \eqref{two-map}.
\end{itemize}
\end{theorem}

\begin{remark}\label{rem:normhypcurve}
In \cite{NP25}, it was observed that the invariant curve obtained in Theorem \ref{NSBTheorem} is normally hyperbolic (see \cite{Fenichel1,hirschpughshub,W94}). Although we do not discuss this property here, it is worth emphasizing that it ensures the persistence of the invariant curve under small perturbations.
\end{remark}

\section{Regularity of the time-$T$ map}\label{section4}

In this section, we investigate the regularity of the time-$T$ map, a fundamental property underlying our approach. To this end, we introduce the $C^r$ functions
\begin{equation*}\label{notation}
F^{j}(t,\bx;\alpha,\e)
:= \sum_{i=1}^{k} \e^{i}\,F_{i}^{j}(t,\bx;\alpha)
   + \e^{k+1}\,R^{j}(t,\bx;\alpha,\e),
\end{equation*}
for $j=0,1,\ldots,n$. With this notation, the solution of the piecewise differential system \eqref{PSDE} can be written in the form
\begin{equation}\label{solutions}
\varphi(t,\bx;\alpha,\e)=
\begin{cases}
\varphi_{0}(t,\bx;\alpha,\e), & 0 \le t \le \tau_{1}(\bx;\alpha,\e),\\[2mm]
\varphi_{1}(t,\bx;\alpha,\e), & \tau_{1}(\bx;\alpha,\e) \le t \le \tau_{2}(\bx;\alpha,\e),\\
\hspace{2cm}\vdots \\[-1mm]
\varphi_{n}(t,\bx;\alpha,\e), & \tau_{n}(\bx;\alpha,\e) \le t \le T,
\end{cases}
\end{equation}
where $\varphi_{0}$ denotes the solution of the Cauchy problem
\begin{equation}\label{PS2}
\begin{cases}
\dot{\bx}(t) = F^{0}(t,\bx;\alpha,\e),\\
\bx(0)=\bx_{0},
\end{cases}
\end{equation}
and, for $j=1,\ldots,n$, $\varphi_{j}$ is defined as the solution of the Cauchy problem
\begin{equation}\label{re}
\begin{cases}
\dot{\bx}(t) = F^{j}(t,\bx;\alpha,\e),\\[1mm]
\bx\!\left(\tau_{j}(\bx;\alpha,\e)\right)
= \varphi_{j-1}\!\left(\tau_{j}(\bx;\alpha,\e),\bx;\alpha,\e\right).
\end{cases}
\end{equation}

Here, $\tau_{j}(\bx;\alpha,\e)$ denotes the switching time at which the trajectory
$\varphi_{j-1}(t,\bx;\alpha,\e)$, starting from
$\varphi_{j-1}(\tau_{j-1}(\bx;\alpha,\e),\bx;\alpha,\e)$ at
$t=\tau_{j-1}(\bx;\alpha,\e)$, intersects the graph of $\theta_{j}$, namely the set
\[
\{(\theta_{j}(\bx),\bx) : \bx \in D\},
\]
see Figure~\ref{SolutionPSDE}. Accordingly, the switching times satisfy the implicit relations
\begin{equation*}\label{flyingtimeeq}
\tau_{j}(\bx;\alpha,\e)
= \theta_{j}\!\left(\varphi_{j-1}\!\left(\tau_{j}(\bx;\alpha,\e),\bx;\alpha,\e\right)\right),
\qquad j=1,\ldots,n,
\end{equation*}
with $\tau_{0}(\bx;\alpha,\e)=0$ and $\tau_{n+1}(\bx;\alpha,\e)=T$.

\begin{figure}[ht!]
\begin{centering}
\begin{minipage}[t]{0.95\textwidth}
       \vspace{0pt}
       \centering
      \begin{overpic}[width=13cm]{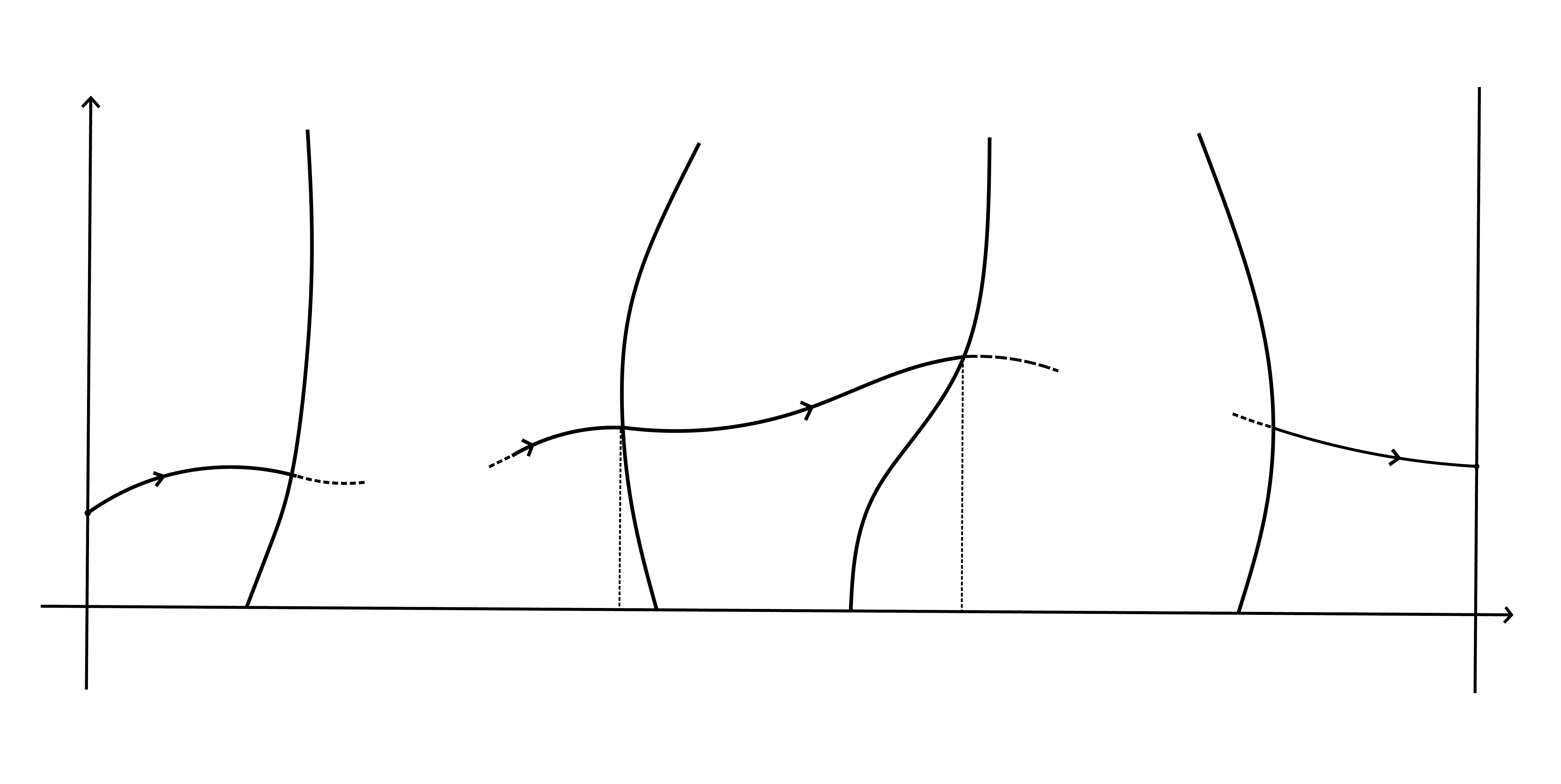}
\put(36,7){\small{$\tau_{j-1}(\bx_{0},\alpha,\e)$}}
\put(58,7){\small{$\tau_{j}(\bx_{0},\alpha,\e)$}}
\put(98,10){\small{$t$}}
\put(93.5,3){\small{$T$}}
\put(5,45){\small{$\bx$}}
\put(2,17){\small{$\bx_{0}$}}
\put(95,20){\tiny{$\varphi(T,\bx_{0};\alpha,\e)$}}
\put(7,22){\tiny{$\varphi_{0}(t,\bx_{0};\alpha,\e)$}}
\put(44,27){\tiny{$\varphi_{j-1}(t,\bx_{0};\alpha,\e)$}}
\put(82.5,23.5){\tiny{$\varphi_{n}(t,\bx_{0};\alpha,\e)$}}
\put(16,43){\small{$\theta_{1}(\bx,\alpha)$}}
\put(29,43){\small{$\cdots$}}
\put(41,43){\small{$\theta_{j-1}(\bx,\alpha)$}}
\put(60,43){\small{$\theta_{j}(\bx,\alpha)$}}
\put(70,43){\tiny{$\cdots$}}
\put(74,43){\small{$\theta_{n}(\bx,\alpha)$}}
\end{overpic}  

\end{minipage}
\caption{\small{Schematic representation of a crossing solution $\varphi$ of the piecewise smooth system} \eqref{PSDE}.}
\label{SolutionPSDE}    
\end{centering}
\end{figure}
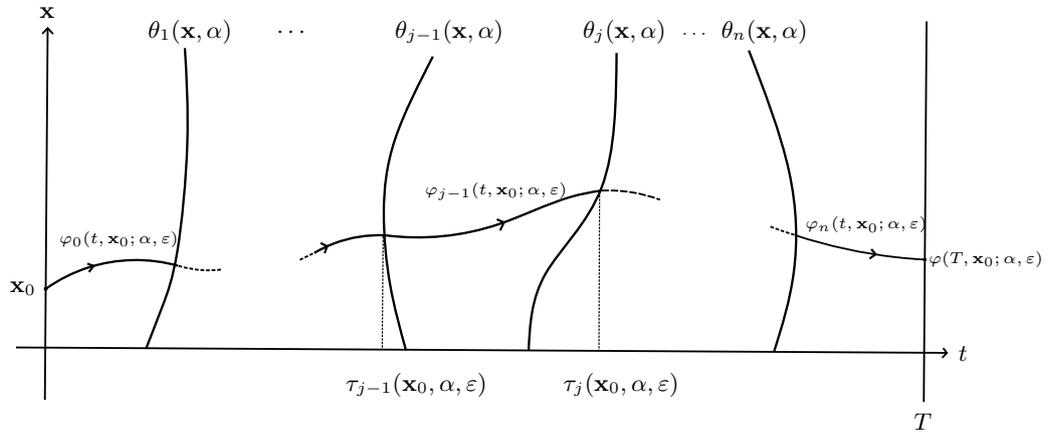 

The time-$T$ map of the piecewise smooth system \eqref{PSDE} is then defined as follows
\begin{equation}\label{Tmap}
\mathcal{P}_T \colon {D} \times {I} \times (-\varepsilon_0,\varepsilon_0) \to \overline{D}, 
\qquad 
\mathcal{P}_T(\boldsymbol{x};\alpha,\varepsilon) := \varphi(T,\boldsymbol{x};\alpha,\varepsilon),
\end{equation}

Note that, by the definition of $\varphi$, 
\[
\mathcal{P}_T(\boldsymbol{x};\alpha,\varepsilon) = \varphi_n(T,\boldsymbol{x};\alpha,\varepsilon).
\] 
Therefore, the regularity of the time-$T$ map~\eqref{Tmap} follows from the regularity of the flow $\varphi_n$. Accordingly, in the following result we show that the time-$T$ map is smooth by proving, via induction, that each function $\varphi_j$ in~\eqref{solutions} is smooth, based on the analysis of the Cauchy problems~\eqref{PS2}--\eqref{re}.

\begin{proposition}\label{reguTmap}
Suppose that, for each $j=0,1,\ldots,n$ and $i=1,\ldots,k$, the functions
\[
F_i^j : \mathbb{S}^1 \times \overline{D} \times \overline{I} \longrightarrow \mathbb{R}^2
\quad\text{and}\quad
R^j : \mathbb{S}^1 \times \overline{D} \times \overline{I} \times (-\varepsilon_0,\varepsilon_0)
\longrightarrow \mathbb{R}^2
\]
appearing in system~\eqref{PSDE} are of class $C^{r}$, for some $r \ge k+1$.  
Then there exists $\varepsilon_1>0$ such that the associated time-$T$ map $\mathcal{P}_T$
is also of class $C^{r}$ on $D \times I \times (-\varepsilon_1,\varepsilon_1).$
\end{proposition}

\begin{proof}
From \eqref{PS2} and \eqref{re}, we have
$$ \frac{\partial\varphi_{j}}{\partial t}(t,\bx; \alpha,\e) = F^{j}(t,\varphi_{j}(t,\bx; \alpha,\e);\alpha,\e), \quad\text{for}\quad j=0,1,\cdots, n,$$
with the initial values
\begin{equation*}\label{recurrence}
\begin{aligned}
\varphi_{0}(0,\bx;\alpha,\e)&=\bx,\\ 
\varphi_{j}(\tau_{j}(\bx;\alpha,\e),\bx;\alpha,\e)&=\varphi_{j-1}(\tau_{j}(\bx;\alpha,\e),\bx;\alpha,\e), \quad\text{for}\quad j=1,2,\cdots,n. 
\end{aligned}    
\end{equation*}
In what follows, we prove by induction that the functions $\varphi_j$,
$j=0,1,\ldots,n$, are of class $C^{r}$.

For $j=0$, $\varphi_0$ is defined as the solution of the initial value problem
\begin{align*}
\frac{\partial \varphi_{0}}{\partial t}(t,\bx;\alpha,\e)
&= F^{0}\bigl(t,\varphi_{0}(t,\bx;\alpha,\e);\alpha,\e\bigr),\\
\varphi_{0}(0,\bx;\alpha,\e) &= \bx .
\end{align*}
Since $F^{0}$ is of class $C^{r}$, the result on differentiable
dependence of solutions on initial conditions and parameters implies that
$\varphi_{0}(t,\bx;\alpha,\e)$ is a $C^{r}$ function.

Assume now, as an inductive hypothesis, that $\varphi_{j-1}$ is of class
$C^{r}$. We prove that $\varphi_{j}$ is also $C^{r}$. To this end, we first show
that the switching time $\tau_{j}$ is a $C^{r}$ function.

Define the function $h_{j}\colon \mathbb{R}\times \ov D\times \ov I\times(-\e_{0},\e_{0}) \longrightarrow
\mathbb{R}$ by
\[
h_{j}(\beta,\bx,\alpha,\e)
= \theta_{j}\bigl(\varphi_{j-1}(\beta,\bx;\alpha,\e)\bigr)-\beta .
\]
For $\e=0$, this reduces to
\[
h_{j}(\beta,\bx,\alpha,0)=\theta_{j}(\bx)-\beta .
\]
Hence, for each $(\bx,\alpha)\in \ov D\times \ov I$,
\[
h_{j}(\theta_{j}(\bx),\bx,\alpha,0)=0,
\qquad
\frac{\partial h_{j}}{\partial \beta}(\theta_{j}(\bx),\bx;\alpha,0)=-1\neq 0 .
\]
Since $\overline D \times \overline I$ is compact, by Implicit Function Theorem, there exist a constant $\varepsilon_{1}\in(0,\varepsilon_{0})$ and a unique function
\[
\xi \colon \overline D \times \overline I \times (-\varepsilon_{1},\varepsilon_{1}) \longrightarrow \mathbb{R}
\]
of class $C^{r}$ satisfying
\[
\xi(\bx,\alpha,0)=\theta_j(\bx)
\quad\text{and}\quad
h_{j}\bigl(\xi(\bx,\alpha,\varepsilon),\bx,\alpha,\varepsilon\bigr)=0,
\]
for all $(\bx,\alpha,\varepsilon)\in \overline D \times \overline I \times (-\varepsilon_{1},\varepsilon_{1})$. 

Since, by definition, $\tau_{j}$ satisfies $\tau_{j}(\bx;\alpha,0)=\theta_j(\bx)$ and $h_{j}\bigl(\tau_{j}(\bx;\alpha,\e),\bx,\alpha,\e\bigr)=0,$
it follows from uniqueness that
\[
\xi (\bx,\alpha,\e)=\tau_{j}(\bx;\alpha,\e),
\quad
\forall\,(\bx,\alpha,\e)\in \ov D\times \ov I\times(-\e_{1},\e_{1}) .
\]
Therefore, we conclude that $\tau_{j}$ is a $C^{r}$ function on $D\times I\times(-\e_{1},\e_{1})$.

Finally, let $\psi_{j}(t,t_{0},\bx_{0},\alpha,\e)$ denote the solution of the
initial value problem
\begin{equation*}\label{FGPC}
\begin{cases}
\bx' = F^{j}(t,\bx,\alpha,\e),\\
\bx(t_{0}) = \bx_{0}.
\end{cases}
\end{equation*}
By the differentiable dependence of solutions on initial conditions and
parameters, $\psi_{j}$ is of class $C^{r}$. Moreover, taking into account that
\begin{equation*}\label{varequality}
\varphi_{j}(t,\bx;\alpha,\e)
=\psi_{j}\!\left(t,\tau_{j}(\bx;\alpha,\e),
\varphi_{j-1}\bigl(\tau_{j}(\bx;\alpha,\e),\bx;\alpha,\e\bigr),
\alpha,\e\right),
\end{equation*}
and since $\tau_{j}$ is of class $C^{r}$ and, by the inductive hypothesis,
$\varphi_{j-1}$ is also $C^{r}$, it follows that $\varphi_{j}$ is of class
$C^{r}$.

This completes the inductive step and proves that $\varphi_j$ is of class $C^r$ for every $j\in\{0,\ldots,n\}$  $D \times I \times (-\varepsilon_1,\varepsilon_1).$ In particular, the time-$T$ map $\mathcal{P}_T$ is also of class $C^{r}$ on $D \times I \times (-\varepsilon_1,\varepsilon_1).$
\end{proof}

\section{Neimark-Sacker Bifurcation for Near-Identity Maps}\label{section5}

Consider an $(\alpha,\varepsilon)$--family of $C^{k}$ two-dimensional near-identity maps
\[
(\bx,\alpha,\varepsilon)\in D \times I \times (-\varepsilon_{0},\varepsilon_{0})
\mapsto\;
\f(\bx;\alpha,\varepsilon),
\]
where
\begin{equation}\label{Identityperturbation}
\f(\bx;\alpha,\varepsilon)
= \bx
+ \varepsilon^{r}\,\f_{r}(\bx;\alpha)
+ \cdots
+\varepsilon^{k}\,\f_{k}(\bx;\alpha)
+ \mathcal{O}(\varepsilon^{k+1}).
\end{equation}
Here, $r$ is a natural number satisfying $1 \le r \le k$, $D\subset\mathbb{R}^{2}$ and
$I\subset\mathbb{R}$ are open and bounded sets, and $\varepsilon_{0}>0$ is a small parameter.

We aim to obtain sufficient conditions guaranteeing that the map
\eqref{Identityperturbation} undergoes a Neimark--Sacker bifurcation, thus ensuring the emergence
of a unique normally hyperbolic invariant closed curve for the map \eqref{Identityperturbation}.

In the first assumption, suppose that the guiding system
\begin{equation}\label{firstsys}
\dot{\bx}=\f_{r}(\bx;\alpha)
\end{equation}
has a Hopf point at $(\bx_{\alpha_{0}},\alpha_{0})\in D\times I$. That is,
$\f_{r}$ admits an equilibrium at $(\bx_{\alpha_{0}},\alpha_{0})$ whose eigenvalues are
$\lambda_{1,2}=\pm i\,b_{0}$, with $b_{0}>0$.

By the Implicit Function Theorem, there exists an open interval around $\alpha_{0}$, which we still denote by $I$, and a smooth curve
$\xi: I \to D$ such that $\xi(\alpha_{0})=\bx_{\alpha_{0}}$ and
\[
\f_{r}(\xi(\alpha);\alpha)=0,
\qquad \forall\,\alpha\in I.
\]
Moreover, the eigenvalues $a(\alpha)\pm i\,b(\alpha)$ of the Jacobian matrix
$D_{\bx}\f_{r}(\bx_{\alpha};\alpha)$ satisfy $a(\alpha_{0})=0$ and $b(\alpha_{0})=b_{0}>0$.
Hence, denoting $\bx_{\alpha}=\xi(\alpha)$ for all $\alpha\in I$, the above discussion can be summarized in the following hypothesis.

\begin{itemize}\label{HypothesisH1}
\item[\textbf{H1.}]
There exists a smooth curve
$\xi:I\to D$, $\xi(\alpha)=\bx_{\alpha}$, defined on a open interval $I$ containing $\alpha_{0}$, such that
$\f_{r}(\bx_{\alpha};\alpha)=0$ for all $\alpha\in I$.
Moreover, the eigenvalues $a(\alpha)\pm i\,b(\alpha)$ of
$D_{\bx}\f_{r}(\bx_{\alpha};\alpha)$ satisfy $a(\alpha_{0})=0$ and $b(\alpha_{0})=b_{0}>0$.
\end{itemize}

Hypothesis \textbf{H1}, together with the Implicit Function Theorem, directly provides directly a smooth curve of fixed points of the map \eqref{Identityperturbation}, as stated in the following lemma.

\begin{lemma}\label{lema2.1}
Suppose that hypothesis \textbf{H1} holds. Then, there exist an interval $J\subset I$ containing $\alpha_{0}$,
a constant $\varepsilon_{1}\in(0,\varepsilon_{0})$, and a unique smooth function
\[
\sigma: J \times (-\varepsilon_{1},\varepsilon_{1}) \longrightarrow D
\]
such that
\[
\f(\sigma(\alpha,\varepsilon);\alpha,\varepsilon)=\sigma(\alpha,\varepsilon),
\qquad
\forall\,(\alpha,\varepsilon)\in J\times(-\varepsilon_{1},\varepsilon_{1}).
\]
\end{lemma}

Our second assumption is a transversality condition.

\begin{itemize}
\item[\textbf{H2.}]
Let $a(\alpha)\pm i\,b(\alpha)$ denote the pair of complex eigenvalues of
$D_{\bx}\f_{r}(\bx_{\alpha};\alpha)$, with $a(\alpha_{0})=0$ and $b(\alpha_{0})=b_{0}>0$.
Assume that
\[
a'(\alpha_{0}) \neq 0.
\]
\end{itemize}

Let $\lambda(\alpha,\varepsilon)$ and $\overline{\lambda(\alpha,\varepsilon)}$ denote the complex
conjugate eigenvalues of the Jacobian matrix
$D_{\bx}\f(\sigma(\alpha,\varepsilon);\alpha,\varepsilon)$, for
$(\alpha,\varepsilon)\in J\times(-\varepsilon_{1},\varepsilon_{1})$.
The following lemma shows that the map \eqref{Identityperturbation}, at the fixed point
$\sigma(\alpha,\varepsilon)$, satisfies Conditions~1 and~2 of
Theorem~\ref{NSBTheorem}.

\begin{lemma}\label{lem2.2}
Assume hypotheses \textbf{H1} and \textbf{H2}. Then there exist
$\varepsilon_{2}\in(0,\varepsilon_{1})$ and a unique $C^{k}$ curve
\[
\beta:(-\varepsilon_{2},\varepsilon_{2})\longrightarrow J,
\qquad \beta(0)=\alpha_{0},
\]
such that, for all $\varepsilon\in(-\varepsilon_{2},\varepsilon_{2})\setminus\{0\}$, the following conditions hold:
\begin{itemize}
\item[1.] $|\lambda(\beta(\varepsilon),\varepsilon)|=1$;
\item[2.] $(\lambda(\beta(\varepsilon),\varepsilon))^{k}\neq 1$ for $k=1,2,3,4$;
\item[3.] $\displaystyle
\frac{d}{d\alpha}|\lambda(\alpha,\varepsilon)|
\bigg|_{\alpha=\beta(\varepsilon)}\neq 0$.
\end{itemize}
\end{lemma}

\begin{proof}
According to Lemma \ref{lema2.1}, there exists a function $\sigma:J\times(-\e_{1},\e_{1})\longrightarrow D$ satisfying $\f(\sigma(\alpha,\e),\alpha,\e)=\sigma(\alpha,\e)$ is guaranteed for all $(\alpha,\e)\in J\times(-\e_{1},\e_{1})$. Also, 
\begin{equation*}
D_{\x}\,\f(\sigma(\alpha,\e);\alpha,\e)= \text{Id} + \e^{r}\,\parde[\f_{r}]{\x}(\x_{\alpha},\alpha)  +\e^{r+1}\,\parde[\f_{r+1}]{\x}(\x_{\alpha},\alpha) +\mathcal{O}(\e^{r+2}),     
\end{equation*}
has the following eigenvalues 
\begin{align*}
\lambda(\alpha,\e)&=1+\e^{r}\,(a(\alpha)+i\,b(\alpha))+\e^{r+1}(c(\alpha)+i\,d(\alpha))+\mathcal{O}(\e^{r+2}),\\
\overline{\lambda(\alpha,\e)}&=1+\e^{r}\,(a(\alpha)-i\,b(\alpha))+\e^{r+1}(c(\alpha)-i\,d(\alpha))+\mathcal{O}(\e^{r+2}).
\end{align*}
Observe 
\begin{align}\label{lambdasqr}
|\lambda(\alpha,\e)|^{2}&=1+\e^{r}\,2a(\alpha)+ \e^{r+1}(a(\alpha)^{2}+b(\alpha)^{2}+2c(\alpha)) +\mathcal{O}(\e^{r+2})\\ \notag
&=1+\e^{r}\,\rho(\alpha,\e)
\end{align}
where $\rho(\alpha,\e)=2a(\alpha)+ \e(a(\alpha)^{2}+b(\alpha)^{2}+2c(\alpha))+\mathcal{O}(\e^{2})$. Moreover, hypothesis \textbf{H2} implies
$$ \rho(\alpha_{0},0)=0 \qquad \text{and} \qquad \frac{\partial \rho}{\partial\alpha}(\alpha_{0},0)=2a'(\alpha_{0}) \neq 0.$$

Hence, applying the Implicit Function Theorem, a unique function $\beta:(-\e_{2},\e_{2})\longrightarrow J$ of class $C^{k}$ is obtained, with $\e_{2} \in (0, \e_{1})$, such that $\beta(0)=\alpha_{0}$ and $\rho(\beta(\e),\e)=0$ throughout the interval $\e\in(-\e_{2},\e_{2})$. Thus, $|\lambda(\beta(\e),\e)|=1$ for all $\e\in(-\e_{2},\e_{2})$, implying that Condition 1 holds.

In additional, 
\begin{align*}
\lambda(\beta(\e),\e)&=1+\e^{r}\,\big(a(\beta(\e))+i\,b(\beta(\e))\big)+\e^{r+1}\big(c(\beta(\e))+i\,d(\beta(\e))\big)+\mathcal{O}(\e^{r+2})\\
    &= 1+\e^{r}\,(i\,b_{0})+\e^{r+1}\big(c(\alpha_{0})+i\,d(\alpha_{0})+a'(\alpha_{0})\beta'(0)+i\,b'(\alpha_{0})\beta'(0)\big)+\mathcal{O}(\e^{r+2})
\end{align*}
with $b_{0}>0$ and for all $\e\in(-\e_{2},\e_{2})$. Thus, since $\e_{2}>0$ can be sufficiently small, it follows
$$ \lambda(\beta(\e),\e)\neq 1, \, \,\lambda(\beta(\e),\e)\neq\pm i, \,\,\text{and}\,\,\lambda(\beta(\e),\e)\neq -\frac{1}{2}\pm i\frac{\sqrt{3}}{2}, $$
for all $\e\in(-\e_{2},\e_{2})\symbol{92} \lbrace{0\rbrace}$. Therefore, $$(\lambda(\beta(\e),\e))\neq 1, \, \,(\lambda(\beta(\e),\e))^{2}\neq 1\,,\,(\lambda(\beta(\e),\e))^{3}\neq 1,\,\,\text{and}\,\, (\lambda(\beta(\e),\e))^{4}\neq 1,$$ this proves Condition 2.

Finally, implicitly deriving \eqref{lambdasqr} at $\alpha=\alpha_{0}$, we get
$$ \frac{\partial}{\partial\alpha}|\lambda(\alpha,\e)|\bigg|_{\alpha=\beta(\e)}= \e^{r}\,a'(\alpha_{0})+\e ^{r+1} ( b(\alpha_{0} ) b'(\alpha_{0} )+ c'(\alpha_{0} ) )+\mathcal{O}(\e^{r+2}). $$

From hypothesis, $a'(\alpha_{0})\neq 0$. Furthermore, we can take $\e_{2}>0$ sufficiently small, if necessary, so that
$$\frac{d}{d\alpha}|\lambda(\alpha,\e)|\bigg|_{\alpha=\beta(\e)}\neq 0,$$
for all $\e\in(-\e_{2},\e_{2})\symbol{92} \lbrace{0\rbrace}$, proving Condition 3.
\end{proof}
\begin{remark}
Notice that hypothesis \textbf{H2} can, in fact, be relaxed. A typical situation occurs when
$ a(\alpha) \equiv 0 $. In this case, it is still possible to obtain a branch of solutions to
$ \rho(\alpha,\e)=0 $ by taking into account higher-order terms of the map~\eqref{Identityperturbation}.
More precisely, if
\[
b(\alpha_0)^2 + 2c(\alpha_0) = 0 \quad \text{and} \quad
b(\alpha_0)\,b'(\alpha_0) + c'(\alpha_0) \neq 0,
\]
then the expansion of $ \rho(\alpha,\e) $ given in~\eqref{lambdasqr} reduces to
\[
\rho(\alpha,\e) = \e\big(b(\alpha)^2 + 2c(\alpha)\big) + \mathcal{O}(\e^2),
\]
which allows the Implicit Function Theorem to be applied once again. As a consequence,
one obtains a $ C^k $ function $ \beta(\e) $ such that
\[
|\lambda(\beta(\e),\e)| = 1 \quad \text{for all } \e \in (-\e_2,\e_2).
\]
This observation highlights a systematic strategy for exploiting higher-order terms when
hypothesis \textbf{H2} fails, thereby opening the door to the analysis of more degenerate
situations.
 \end{remark}

Applying the change of variable $\bx=\by+\sigma(\alpha,\varepsilon)$ and parameter  $\alpha=\tau+\beta(\varepsilon)$ to the map \eqref{Identityperturbation}, we obtain
\begin{equation}\label{Identityperchange}
\begin{aligned}
\by \mapsto \boldsymbol{\mathcal{H}}_{\varepsilon}(\by,\tau)
&=\f\!\left(\by+\sigma(\tau+\beta(\varepsilon),\varepsilon);\,
\tau+\beta(\varepsilon),\varepsilon\right)
-\sigma\!\left(\tau+\beta(\varepsilon),\varepsilon\right)\\
&=\by+\varepsilon^{r}\,
\boldsymbol{F}\!\left(\by+\sigma(\tau+\beta(\varepsilon),\varepsilon);\,
\tau+\beta(\varepsilon),\varepsilon\right),
\end{aligned}
\end{equation}
where
\[
\boldsymbol{F}(\bx;\alpha,\varepsilon)
=\f_{r}(\bx;\alpha)
+\varepsilon\,\f_{r+1}(\bx;\alpha)
+\cdots
+\varepsilon^{k}\,\f_{k}(\bx;\alpha)
+\varepsilon^{k-r+1}\,\widetilde{\boldsymbol{F}}(\bx;\alpha,\varepsilon).
\]
By construction, $\by=0$ is a fixed point of the map \eqref{Identityperchange}, and $\tau=0$ corresponds to the critical value of the parameter.

Expanding $\boldsymbol{\mathcal{H}}_{\varepsilon}(\by,0)$ in a neighborhood of $\by=0$, we obtain
$$ \boldsymbol{\mathcal{H}}_{\e}(\y,0)= \A_{\e}\y + \frac{1}{2}\,\B_{\e}(\y,\y)+\frac{1}{6}\,\mathcal{C}_{\e}(\y,\y,\y)+\mathcal{O}(||\y||^{4}),  $$
where $\A_{\e}=D_{\x}\mathcal{H}_{\e}(0,0)$. Moreover,
\begin{equation}\label{multilinearfunctions}
\A_{\e}=\text{Id}+\e^{r}\,A_{\e}+ \mathcal{O}(\e^{k+1}), \quad \B_{\e}=\e^{r}\,B_{\e}+ \mathcal{O}(\e^{k+1}), \quad \mathcal{C}_{\e}=\e^{r}\,C_{\e}+ \mathcal{O}(\e^{k+1}),
\end{equation}
with
\begin{equation}\label{multilinearfunctions2}
\begin{aligned}
A_{\e}&= A_{r}+\e\,A_{r+1}+\cdots +\e^{k-r}\,A_{k},\\
B_{\e}(\bu,\bv)&= B_{r}(\bu,\bv)+\e\,B_{r+1}(\bu,\bv)+\cdots +\e^{k-r}\,B_{k}(\bu,\bv),\\
C_{\e}(\bu,\bv,\w)&= C_{r}(\bu,\bv,\w)+\e\,C_{r+1}(\bu,\bv,\w)+\cdots +\e^{k-r}\,C_{k}(\bu,\bv,\w).
\end{aligned}
\end{equation}

Our final hypothesis is a technical assumption that allows us to standardize the choice of the eigenvectors $\mathbf{p}$ and $\mathbf{q}$ in \eqref{firstLyapunovcoefficient} as $\q=\bp=(1,-i)/\sqrt{2}$, thereby simplifying the computations. Notice that this assumption is not restrictive and can always be enforced by a linear change of coordinates.

\begin{itemize}
\item[\textbf{H3.}]
Let $\widetilde a_{\varepsilon}\pm i\,\widetilde b_{\varepsilon}$ denote the eigenvalues of $\e^r A_{\varepsilon}$. Assume that
\[
\mathrm{Id}+\varepsilon^{r}A_{\varepsilon}
=
\begin{pmatrix}
1+\widetilde a_{\varepsilon} & -\widetilde b_{\varepsilon}\\
\widetilde b_{\varepsilon} & 1+\widetilde a_{\varepsilon}
\end{pmatrix}.
\]
\end{itemize}

We now proceed to verify Conditions 1--3 of Theorem~\ref{NSBTheorem} for the map~\eqref{Identityperchange}.

Observe that
\[
\kappa_{\varepsilon}(\tau)
=\lambda(\tau+\beta(\varepsilon),\varepsilon),
\qquad
\overline{\kappa}_{\varepsilon}(\tau)
=\overline{\lambda(\tau+\beta(\varepsilon),\varepsilon)}
\]
are the eigenvalues of $D_{\by}\boldsymbol{\mathcal{H}}_{\varepsilon}(0,\tau)$. Writing
\[
\kappa_{\varepsilon}(\tau)
=r_{\varepsilon}(\tau)\,e^{i\varphi_{\varepsilon}(\tau)},
\qquad
\varphi_{\varepsilon}(0)=\theta_{\varepsilon}\in(0,\pi),
\]
we have
\[
\lambda(\beta(\varepsilon),\varepsilon)
=\kappa_{\varepsilon}(0)
=r_{\varepsilon}(0)\,e^{i\theta_{\varepsilon}}.
\]
From Lemma~\ref{lem2.2}, it follows that
\[
|\lambda(\beta(\varepsilon),\varepsilon)|=1,
\qquad
\forall\,\varepsilon\in(0,\varepsilon_2),
\]
and, therefore,
\begin{equation}\label{conditionscheck1}
r_{\varepsilon}(0)=1,
\qquad
\forall\,\varepsilon\in(0,\varepsilon_2).
\end{equation}
Moreover, since
\[
(\lambda(\beta(\varepsilon),\varepsilon))^{k}
=e^{ik\theta_{\varepsilon}},
\]
Lemma~\ref{lem2.2} implies
\begin{equation}\label{conditionscheck2}
e^{ik\theta_{\varepsilon}}\neq 1,
\qquad
\forall\,\varepsilon\in(0,\varepsilon_2),
\quad k=1,2,3,4.
\end{equation}
Analogously, using
\[
\frac{d}{d\tau}\kappa_{\varepsilon}(\tau)\bigg|_{\tau=0}
=
\frac{d}{d\alpha}\lambda(\alpha,\varepsilon)
\bigg|_{\alpha=\beta(\varepsilon)},
\]
and again Lemma~\ref{lem2.2}, we obtain
\begin{equation}\label{conditionscheck3}
\frac{d}{d\tau}r_{\varepsilon}(\tau)\bigg|_{\tau=0}\neq 0,
\qquad
\forall\,\varepsilon\in(0,\varepsilon_2).
\end{equation}
Equations \eqref{conditionscheck1}--\eqref{conditionscheck3} show that the map
\eqref{Identityperchange} satisfies Conditions~1 and~2 of Theorem~\ref{NSBTheorem}.

Now, expanding the eigenvalues around $\e=0$, we obtain
\begin{equation}\label{expreeigenIdentity}
e^{i\theta_{\varepsilon}}
=
1+\sum_{j=r}^{k}\varepsilon^{j}(a_{j}+i\,b_{j})
+\mathcal{O}(\varepsilon^{k+1}).
\end{equation}
From hypothesis \textbf{H1}, we have $a_{r}=0$ and $b_{r}=b_{0}>0$ and, therefore,
\[
\widetilde a_{\varepsilon}
=\sum_{j=r+1}^{k}\varepsilon^{j}a_{j},
\qquad
\widetilde b_{\varepsilon}
=\sum_{j=r}^{k}\varepsilon^{j}b_{j}.
\]
By Hypothesis {\bf H3},  we may take
$\q=\bp=(1,-i)/\sqrt{2}$ in \eqref{firstLyapunovcoefficient}, yielding
\begin{equation}\label{lyapuforIdentity}
\begin{aligned}
{\ell_{1}^{\e}}= \textrm{Re}\left(\frac{e^{-i\theta_{\e}}\big<\mathbf{q},\e^{r}\,C_{\e}(\mathbf{q},\mathbf{q},\bar{\mathbf{q}})\big>}{2} \right) -& \textrm{Re}\left(\frac{(1-2\,e^{i\theta})\,e^{-2i\theta}}{2(1-e^{i\theta})}\,\big<\mathbf{q},\e^{r}\,B_{\e}(\mathbf{q},\mathbf{q})\big>\big<\mathbf{q},\e^{r}\,B_{\e}(\mathbf{q},\bar{\mathbf{q}})\big>\right)\notag\\ 
-&\frac{1}{2}\big|\big<\mathbf{q},\e^{r}\,B_{\e}(\mathbf{q},\bar{\mathbf{q}})\big>\big|^{2} -\frac{1}{4}\big|\big<\mathbf{q},\e^{r}\,B_{\e}(\bar{\mathbf{q}},\bar{\mathbf{q}})\big>\big|^{2}.
\end{aligned}
\end{equation} 
Thus, the Taylor expansion of $\ell_{1}^{\varepsilon}$ around $\varepsilon=0$ reads
\begin{equation}\label{l1forIdentity}
\ell_{1}^{\varepsilon}
=\varepsilon^{r}\ell_{1,r}
+\varepsilon^{r+1}\ell_{1,r+1}
+\cdots
+\varepsilon^{k}\ell_{1,k}
+\mathcal{O}(\varepsilon^{k+1}),
\end{equation}
where $\varepsilon^{j}\ell_{1,j}$, $r\le j\le k$, is called the $j$th-order Lyapunov coefficient.

The reader is referred to the Appendix, where the expansion of $\ell_{1}^{\varepsilon}$ is computed explicitly up to second order in the case $k = 2$ and $r = 1$.

Now, we can state the first main result of this paper:

\begin{mtheorem}\label{TheoremA}
Assume that hypotheses \textbf{H1}, \textbf{H2}, and \textbf{H3} hold, and that $\ell_{1,j}\neq 0$ for some $j\in\{r,\ldots,k\}$. Let $s\in\{r,\ldots,k\}$
be the smallest index such that $\ell_{1,s}\neq 0$. Then, for every $\e>0$ sufficiently small, there exist a $C^{1}$ curve $\beta(\e)\in J$, with $\beta(0)=\alpha_{0}$, and open neighborhoods $U_{\e}\subset \R^{2}$ of the fixed point $\x=\sigma(\alpha,\e)$ and $J_{\e}\subset J$ of $\beta(\e)$ satisfying 
\begin{itemize}
\item[\textbf{{1.}}] When $\alpha\in J_{\e}$ and $\ell_{1,s}<0$, we have the following possibilities: if $(\alpha-\beta(\e))\,a'(\alpha_{0})\leq 0$, then $\x=\sigma(\alpha,\e)$ is attracting; if $(\alpha-\beta(\e))\,a'(\alpha_{0})>0$, then $\x=\sigma(\alpha,\e)$ is repelling and there exists a unique attracting normally hyperbolic closed invariant curve $\Gamma_{\alpha,\e}$ in $U_{\e}$.
\item[\textbf{{2.}}] When $\alpha\in J_{\e}$ and $\ell_{1,s}>0$, we have the following possibilities: if $(\alpha-\beta(\e))\,a'(\alpha_{0})\geq 0$, then $\x=\sigma(\alpha,\e)$ is repelling; if $(\alpha-\beta(\e))\,a'(\alpha_{0})<0$, then $\x=\sigma(\alpha,\e)$ is attracting and there exists a unique repelling normally hyperbolic closed invariant curve $\Gamma_{\alpha,\e}$ in $U_{\e}$.
 \end{itemize}
\end{mtheorem}

\begin{proof}
Notice that for proving this theorem it is sufficient to show that map \eqref{Identityperchange} satisfies the three conditions of Theorem \ref{NSBTheorem}. However, according to equations \eqref{conditionscheck1}, \eqref{conditionscheck2} and \eqref{conditionscheck3}, map \eqref{Identityperchange} already meets Conditions 1 and 2 of Theorem \ref{NSBTheorem}. 

Now, to proof Condition 3, consider $\mathbf{p}_{\e}\in \C^{2}$ the eigenvector of the matrix $\A_{\e}^{T}$ with respect to the eigenvalue $e^{-i\theta_{\e}}$ and $\mathbf{q}_{\e}\in\C^{2}$ the eigenvector of the matrix  $\A_{\e}$ with respect to the eigenvalue $e^{i\theta_{\e}}$ such that $\langle\mathbf{p}_{\e}, \mathbf{q}_{\e}\rangle=1$. Furthermore, we claim that $\mathbf{p}_{\e}=\widetilde{\mathbf{p}}_{\e}+\mathcal{O}(\e^{k-r+1})$ and $\mathbf{q}_{\e}=\widetilde{\mathbf{q}}_{\e}+\mathcal{O}(\e^{k-r+1})$, where $\widetilde{\mathbf{p}}_{\e}, \widetilde{\mathbf{q}}_{\e}\in \C^{2}$ satisfy 
$$ A_{\e}^{T}\widetilde{\mathbf{p}}_{\e}=(\widetilde{a}_{\e}-i\widetilde{b}_{\e})\widetilde{\mathbf{p}}_{\e},\quad A_{\e}\widetilde{\mathbf{q}}_{\e}=(\widetilde{a}_{\e}+i\widetilde{b}_{\e})\widetilde{\mathbf{q}}_{\e}, \,\,\text{and}\,\, \langle\widetilde{\mathbf{p}}_{\e},\widetilde{\mathbf{q}}_{\e}\rangle=1.$$
Indeed, if $\bv=\widetilde{\bv}+\mathcal{O}(\e^{k-r+1})\in\C^{2}$ is an eigenvector of $\A_{\e}$ related to $\kappa_{\e}(0)$, then
$$ [\A_{\e}-\kappa_{\e}(0)\,\text{Id}]\bv = 0.$$

Using equation \eqref{expreeigenIdentity}, the above equality writes 
$$ [\text{Id}+\e^{r}\,A_{\e}+ \mathcal{O}(\e^{k+1})-(1+\widetilde{a}_{\e}+i\,\widetilde{b}_{\e}+\mathcal{O}(\e^{k+1}))\,\text{Id}]\bv = 0.$$

From {\bf H3}, we get
$$\begin{pmatrix}
1+\widetilde{a}_{\e} & -\widetilde{b}_{\e} \\
\widetilde{b}_{\e} & 1+\widetilde{a}_{\e}
\end{pmatrix}\begin{pmatrix}
\frac{1}{\sqrt{2}}  \\
\frac{-i}{\sqrt{2}}
\end{pmatrix}=\frac{1}{\sqrt{2}}\begin{pmatrix}
1+\widetilde{a}_{\e}+i\,\widetilde{b}_{\e}  \\
\widetilde{b}_{\e}-i(1+\widetilde{a}_{\e}) 
\end{pmatrix}=(1+\widetilde{a}_{\e}+i\,\widetilde{b}_{\e} )\begin{pmatrix}
\frac{1}{\sqrt{2}}  \\
\frac{-i}{\sqrt{2}}
\end{pmatrix}\cdot$$

Thus, taking $\widetilde{\bv}=(1/\sqrt{2},-i/\sqrt{2})$, it follows that $\widetilde{\bv}$ is an eigenvector of $A_{\e}$ with respect to the eigenvalue $\widetilde{a}_{\e}+i\,\widetilde{b}_{\e}$. In a similar way, we can show the same for the matrix $A_{\e}^{T}$. 

Therefore, we can take $\widetilde{\mathbf{p}}_{\e}= \widetilde{\mathbf{q}}_{\e}=\q$, so from equation \eqref{expressionsg_ij's}, we have
\begin{align*}
 g_{11}&=\left<\mathbf{q},\B_{\e}(\mathbf q,\mathbf {\bar{q}})\right>=\left<\mathbf q, \e^{r}\,B_{\e}(\mathbf q,\mathbf {\bar{q}})\right>+\mathcal{O}(\e^{k+1}),\\
 g_{20}&=\left<\mathbf{q},\B_{\e}(\mathbf q,\mathbf {q})\right>=\left<\mathbf q, \e^{r}\,B_{\e}(\mathbf q,\mathbf {q})\right>+\mathcal{O}(\e^{k+1}),\\
 g_{02}&=\left<\mathbf q,\B_{\e}(\mathbf {\bar{q}},\mathbf {\bar{q}})\right>=\left<\mathbf q, \e^{r}\,B_{\e}(\mathbf {\bar{q}},\mathbf {\bar{q}})\right>+\mathcal{O}(\e^{k+1})\\
g_{21}&=\left<\mathbf q, \mathcal{C}_{\e}(\mathbf q,\mathbf q, \mathbf {\bar{q}})\right>= \left<\mathbf q, \e^{r}\,C_{\e}(\mathbf q,\mathbf q, \mathbf {\bar{q}})\right> +\mathcal{O}(\e^{k+1}). 
\end{align*}

Conversely, from hypothesis, we denote by $s$ with $r\leq s\leq k$ the first subindex such that $\ell_{1,s}^{\e}\neq 0$. Thus, substituting the above expressions into the formula of the first Lyapunov coefficient \eqref{firstLyapunovcoefficient}, we get
\begin{align*}
\ell_{1}&=\ell_{1}^{\e} + \mathcal{O}(\e^{k+1})\\
&=\e^{s}\,\ell_{1,s}+\e^{s+1}\,\ell_{1,s+1}+\cdots+\e^{k-r}\,\ell_{1,k}+\mathcal{O}(\e^{k+1}).
\end{align*}
Moreover, since $\ell_{1,s}^{\e}\neq 0$, there exists $\e^{*}\in(0,\e_{2})$, satisfying
$$\text{signal}\,(\ell_{1}(\e))=\text{signal}\,(\ell_{1,s}),$$ for all $\e\in(0,\e^{*})$. 

Hence, Theorem \ref{NSBTheorem} and Remark \ref{rem:normhypcurve} can be applied to map \eqref{Identityperchange}, for all $\e\in(0,\e^{*})$, thereby ensuring the existence of open sets $V_{\e}\subset \R^{2}$ around $\y=0$ and $\widetilde{I}_{\e}\subset I_{\e}$ surrounding $\tau=0$ such that statements 1 and 2 of Theorem \ref{NSBTheorem} are satisfied. 
Hence, returning to the change of coordinates in equation \eqref{Identityperchange}, we can find an open set $U_{\e}\subset D$ around $\x=\sigma(\beta(\e),\e)$ and an open interval $J_{\e}$ around $\beta(\e)$ such that
\begin{itemize}
\item[\textbf{{1.}}] When $\alpha\in J_{\e}$ and $\ell_{1,s}<0$, we have the following possibilities: if $(\alpha-\beta(\e))\,a'(\alpha_{0})\leq 0$, then $\x=\sigma(\alpha,\e)$ is attracting; in contrast, if $(\alpha-\beta(\e))\,a'(\alpha_{0})>0$, then $\x=\sigma(\alpha,\e)$ is repelling and there exists a unique asymptotically attracting closed invariant curve $\Gamma_{\alpha,\e}$ in $U_{\e}$.
\item[\textbf{{2.}}] When $\alpha\in J_{\e}$ and $\ell_{1,s}>0$, we have the following possibilities: if $(\alpha-\beta(\e))\,a'(\alpha_{0})\geq 0$, then $\x=\sigma(\alpha,\e)$ is repelling; in contrast, if $(\alpha-\beta(\e))\,a'(\alpha_{0})<0$, then $\x=\sigma(\alpha,\e)$ is attracting and there exists a unique repelling closed invariant curve $\Gamma_{\alpha,\e}$ in $U_{\e}$.
 \end{itemize}
\end{proof}

Notice that when $r=s$ in Theorem \ref{TheoremA}, i.e., when $s$ is the first subindex such that both $\f_{s}\neq 0$ and $\ell_{1,s}\neq 0$, this hypothesis together with \textbf{H1} and \textbf{H2}, imply a Hopf bifurcation in the following system
\begin{equation}\label{ssystem}
 \dot{\x}= \f_{s}(\x;\alpha).    
\end{equation}
Thus, in such a case, Theorem \ref{TheoremA} tells us that a Hopf bifurcation of system \eqref{ssystem} determines the existence of a unique invariant closed curve of the map \eqref{Identityperturbation} emerging from a fixed point. However, more than that, Theorem \ref{TheoremA} also generalizes the previous result including the case when $r\neq s$ which could happen in many applications. 

\section{Torus bifurcation in nonsmooth differential equations}\label{section6}

This section aims to interpret Theorem~\ref{TheoremA} in the setting of the piecewise smooth systems \eqref{PSDE} introduced in Section~\ref{section1}.

To derive conditions for the occurrence of a Neimark--Sacker bifurcation in the associated time-$T$ map \eqref{Tmap}, we combine the regularity result of Proposition~\ref{reguTmap} with the expansion obtained via the Melnikov procedure \eqref{mapintroT}. Specifically, the time-$T$ map is a $C^r$ near-identity map, $r\geq k$, which admits the expansion
\begin{equation*}
  \mathcal{P}_{T}(\bx;\alpha,\varepsilon)
  = \bx
  + \varepsilon\,\Delta_{1}(\bx;\alpha)
  + \cdots
  + \varepsilon^{k}\Delta_{k}(\bx;\alpha)
  + \mathcal{O}(\varepsilon^{k+1}), \quad (\bx,\alpha,\e)\in \ov D\times \ov I\times (-\e_0,\e_0),
\end{equation*}
where $\Delta_{i}$, $i\in\{1,\ldots,k\},$ are the Melnikov functions. Also, let $r\in\{1,2,\ldots,k\}$ be the smallest index for which
$\Delta_{1}=\Delta_{2}=\cdots=\Delta_{r-1}=0,$ and $\Delta_{r}\neq 0.$
Under this assumption, the expansion simplifies to
\begin{equation}\label{Tmap3}
  \mathcal{P}_{T}(\bx;\alpha,\varepsilon)
  = \bx
  + \varepsilon^{r}\Delta_{r}(\bx;\alpha)
  + \cdots
  + \varepsilon^{k}\Delta_{k}(\bx;\alpha)
  + \mathcal{O}(\varepsilon^{k+1}) \quad (\bx,\alpha,\e)\in \ov D\times \ov I\times (-\e_0,\e_0).
\end{equation}

Theorem~\ref{TheoremA} can therefore be applied to the reduced form \eqref{Tmap3}, providing sufficient conditions for the bifurcation of a normally hyperbolic invariant closed curve. As a result, under such conditions, the piecewise smooth system \eqref{PSDE} possesses an invariant torus that persists under sufficiently small perturbations. Accordingly, in the notation of Section \ref{section5}, we set
\[
\f(\bx;\alpha,\varepsilon)=\mathcal{P}_{T}(\bx;\alpha,\varepsilon),
\qquad
\f_i(\bx;\alpha)=\Delta_i(\bx;\alpha), \quad i\in\{r,\ldots,k\}.
\]
With these identifications, and noting that fixed points and invariant closed curves of $\mathcal{P}_{T}$ correspond, respectively, to periodic solutions and invariant tori of \eqref{PSDE}, Theorem~\ref{TheoremA} yields the following result:

\begin{mtheorem}\label{MainTheorem}
Let $r\in\lbrace{1,2,\cdots,k\rbrace}$ be such that $\Delta_{1}=\cdots=\Delta_{r-1}=0$ but $\Delta_{r}\neq 0$. 
Assume that hypotheses \textbf{H1}, \textbf{H2}, and \textbf{H3} hold, and that $\ell_{1,j}\neq 0$ for some $j\in\{r,\ldots,k\}$. Let $s\in\{r,\ldots,k\}$
be the smallest index such that $\ell_{1,s}\neq 0$. Then, for every $\e>0$ sufficiently small, there exist a smooth curve $\beta(\e)\in J$, with $\beta(0)=\alpha_{0}$, and open sets $\mathcal{U}_{\e}\subset S^{1}\times D$ of the $T$-periodic solution  $\Psi(t,\beta(\e),\e)$ and $J_{\e}\subset J$ around $\beta(\e)$ satisfying the following assertions 

\begin{itemize}
\item[\textbf{{1.}}] When $\alpha\in J_{\e}$ and  $\ell_{1,s}<0$, we have the following possibilities: if  $(\alpha-\beta(\e))\,a'(\alpha_{0})\leq 0$, then $\Psi(t,\beta(\e),\e)$ is attracting; if $(\alpha-\beta(\e))\,a'(\alpha_{0})>0$, then $\Psi(t,\beta(\e),\e)$ is repelling and the piecewise smooth system \eqref{PSDE} has a unique attracting invariant tori $T_{\alpha,\e}$ in $\mathcal{U}_{\e}$, which persists under small perturbation.
\item[\textbf{{2.}}] When $\alpha\in J_{\e}$ and $\ell_{1,s}>0$, we have the following possibilities: if $(\alpha-\beta(\e))\,a'(\alpha_{0})\geq 0$, then $\Psi(t,\beta(\e),\e)$ is repelling; if $(\alpha-\beta(\e))\,a'(\alpha_{0})<0$, then $\Psi(t,\beta(\e),\e)$ is attracting and the piecewise smooth system \eqref{PSDE} has a unique repelling invariant tori $T_{\alpha,\e}$ in $\mathcal{U}_{\e}$, which persists under small perturbation.
 \end{itemize}
\end{mtheorem}

Observe that when we have $s=r$ in Theorem \ref{MainTheorem}, then Hypotheses \textbf{H1}, \textbf{H2}, and $\ell_{1,r}\neq 0$ ensure that system \eqref{firstsys} undergoes a Hopf bifurcation.  Consequently, Theorem \ref{MainTheorem} states that a Hopf bifurcation in the system \eqref{firstsys}, implies the emerging of an isolated invariant tori in the piecewise smooth system \eqref{PSDE}.

\section{Invariant torus in a 3D piecewise linear system} \label{AppliTorus}

To exemplify the application of Theorem \ref{MainTheorem}, the following class of 3D piecewise linear systems is considered:
\begin{equation}\label{ds0}
	\begin{pmatrix}
x' \\
y' \\
z'
\end{pmatrix}= X(x,y,z,\e)=\left\{
\begin{aligned}   
	X^{+}(x,y,z)=&\begin{pmatrix}
-y \\
x \\
0
\end{pmatrix} +\e\, A^{+}\,\begin{pmatrix}
x \\
y \\
z
\end{pmatrix},\quad y>0\\
	X^{-}(x,y,z)=&\begin{pmatrix}
-y \\
x \\
0
\end{pmatrix}+\e\, A^{-}\,\begin{pmatrix}
x \\
y \\
z
\end{pmatrix}+\e^2 B^{-}\,\begin{pmatrix}
x \\
y \\
z
\end{pmatrix}, \quad y<0
\end{aligned}
\right.
\end{equation}
where
\begin{equation*}
A^{+}=\begin{pmatrix}
	0 & 0 & -\dfrac{4 \pi ^2 b}{8+9 \pi ^2} \\
	0 & -1 & -1 \\
	0 & 0 & \dfrac{1}{2}
\end{pmatrix},
\qquad
A^{-}=\begin{pmatrix}
	\alpha & 4 & -1 \\
	0 & 0 & 0 \\
	0 & -1 & 0
\end{pmatrix},
\qquad
B^{-}=\begin{pmatrix}
	0 & 0 & 0 \\
	0 & 0 & b \\
	0 & 0 & 0
\end{pmatrix}.
\end{equation*}
Here, $b\neq 0$ and $\alpha$ denotes the bifurcation parameter. Using the results from Sections \ref{section5} and \ref{lyapcoefficients}, we derive conditions ensuring the presence of an isolated invariant torus in the phase space of system \eqref{ds0}. Our approach consists into rewriting system \eqref{ds0} in the standard form \eqref{PSDE} and, after an appropriate change of variables, showing that the associated time-$T$ map fulfills all the assumptions of Theorem \ref{TheoremA}. The statement of this result is presented in the theorem below.

\begin{mtheorem}\label{mainresult} For every $\e>0$ sufficiently small there exist a unique limit cycle $\Psi(t,\alpha,\e)$ of the 3D piecewise linear system \eqref{ds0}. In addition, assuming that $b\neq 0$, there exist a smooth curve $\beta(\e)$ fulfilling
 $$\beta(\e)=-\e\left(\dfrac{24 b}{8+9 \pi ^{2}}-1\right)+\mathcal{O}(\e^2),$$
 and open sets $\mathcal{U}_{\e}\subset \R^{3}$ of the limit cycle $\Psi(t,\beta(\e),\e)$ and $J_{\e}$ around $\beta(\e)$ such that the following assertions hold.
 
 \item[\textbf{{1.}}] When $b>0$ and $\alpha\in J_{\e}$, we have the following possibilities: if  $\alpha-\beta(\e)\leq 0$, then $\Psi(t,\beta(\e),\e)$ is asymptotically attracting; if $\alpha-\beta(\e)>0$, then $\Psi(t,\beta(\e),\e)$ is repelling and the 3D piecewise linear system \eqref{ds0} has a unique attracting invariant torus in $\mathcal{U}_{\e}$, which persists under small perturbation.
\item[\textbf{{2.}}] When $b<0$ and $\alpha\in J_{\e}$, we have the following possibilities: if $\alpha-\beta(\e)\geq 0$, then $\Psi(t,\beta(\e),\e)$ is repelling; if $\alpha-\beta(\e)<0$, then $\Psi(t,\beta(\e),\e)$ is asymptotically attracting and the 3D piecewise linear system \eqref{ds0} has a unique repelling invariant torus in $\mathcal{U}_{\e}$, which persists under small perturbation.
\end{mtheorem}

\begin{proof}
First, we transform system~\eqref{ds0} into the form~\eqref{PSDE} by means of the cylindrical change of coordinates,
\[
(x,y,z)=(r\cos\theta,r\sin\theta,z).
\]
In these coordinates, system~\eqref{ds0} becomes
\begin{equation}\label{ds1}
\begin{pmatrix}
r'\\
\theta'\\
z'
\end{pmatrix}
=
\begin{cases}
X_{*}^{+}(r,\theta,z,\varepsilon), & \theta>\pi,\\[2mm]
X_{*}^{-}(r,\theta,z,\varepsilon), & \theta<\pi,
\end{cases}
\end{equation}
where
\begin{align*}
X_{*}^{+}(r,\theta,z,\varepsilon)
&=
\begin{pmatrix}
\displaystyle
\frac{1}{4}\varepsilon \sin\theta\big(-4r\sin\theta-4z+\pi^{2}+4\big)
-\frac{4\pi^{2}bz\varepsilon\cos\theta}{8+9\pi^{2}}
\\[2mm]
\displaystyle
1+\frac{4\pi^{2}bz\varepsilon\sin\theta}{(8+9\pi^{2})r}
+\frac{\varepsilon\cos\theta\big(-4r\sin\theta-4z+\pi^{2}+4\big)}{4r}
\\[2mm]
\displaystyle
\frac{1}{2}\varepsilon(z-5)
\end{pmatrix},\\[3mm]
X_{*}^{-}(r,\theta,z,\varepsilon)
&=
\begin{pmatrix}
\displaystyle
\frac{1}{2}\varepsilon\big(\alpha r\cos(2\theta)+\alpha r+4r\sin(2\theta)-2z\cos\theta\big)
+2bz\varepsilon^{2}\sin\theta
\\[2mm]
\displaystyle
1+\frac{\varepsilon\cos\theta\big(bz\varepsilon-\alpha r\sin\theta\big)+\varepsilon\sin\theta\big(z-4r\sin\theta\big)}{r}
\\[2mm]
\displaystyle
-r\varepsilon\sin\theta
\end{pmatrix}.
\end{align*}

Since $\theta'>0$, we can take $\theta$ as the new independent variable. Computing the quotients $r'/\theta'$ and $z'/\theta'$ in~\eqref{ds1} and expanding the resulting expressions around $\varepsilon=0$, we obtain the reduced system, which is written in the form \eqref{PSDE},
\begin{equation}\label{ds2}
(\dot r,\dot z)
=
\varepsilon F_{1}(\theta,r,z)
+\varepsilon^{2}F_{2}(\theta,r,z)
+\mathcal{O}(\varepsilon^{3}),
\end{equation}
where
\begin{align*}
F_{1}^{+}(\theta,r,z)
&=
\left(
\frac{1}{4}\sin\theta\big(-4r\sin\theta-4z+\pi^{2}+4\big)
-\frac{4\pi^{2}bz\cos\theta}{8+9\pi^{2}},
\,
\frac{z-5}{2}
\right),\\
F_{1}^{-}(\theta,r,z)
&=
\left(
\frac{1}{2}\big(\alpha r\cos(2\theta)+\alpha r+4r\sin(2\theta)-2z\cos\theta\big),
\,
-r\sin\theta
\right),
\end{align*}
and
\begin{align*}
F_{2}^{+}(\theta,r,z)
&=
\Bigg(
\frac{16\pi^{2}bz\sin\theta+(8+9\pi^{2})\cos\theta(-4r\sin\theta-4z+\pi^{2}+4)}
{16(8+9\pi^{2})^{2}r}
\\
&\qquad\times
\Big(16\pi^{2}bz\cos\theta
+(8+9\pi^{2})\sin\theta(4r\sin\theta+4z-\pi^{2}-4)\Big),
\\
&\qquad
\frac{z-5}{8(8+9\pi^{2})r}
\Big((8+9\pi^{2})\cos\theta(4r\sin\theta+4z-\pi^{2}-4)
-16\pi^{2}bz\sin\theta\Big)
\Bigg),\\[2mm]
F_{2}^{-}(\theta,r,z)
&=
\Bigg(
\frac{\sin(2\theta)\big(-\alpha r\cos\theta-4r\sin\theta+z\big)^{2}}{2r}
+bz\sin\theta,
\\
&\qquad
\sin^{2}\theta\big(-\alpha r\cos\theta-4r\sin\theta+z\big)
\Bigg).
\end{align*}

Observe that the time-$T$ map associated with~\eqref{ds2} can be written as
\begin{equation*}
\mathcal{P}_{T}(r,z;\alpha,\varepsilon)
:= \x
+ \varepsilon\,\Delta_{1}(r,z;\alpha)
+ \varepsilon^{2}\,\Delta_{2}(r,z;\alpha)
+ \mathcal{O}(\varepsilon^{3}),\quad (r,z,\alpha,\e)\in \ov D\times\ov I\times (-\e_0,\e_0),
\end{equation*}
where the Melnikov functions $\Delta_{1}$ and $\Delta_{2}$ are given by~\eqref{mel1}. We may also take 
\[
D=(-20,20)\times(-20,20) \quad\text{and}\quad I=(-1/2,1/2).
\] 

In the present setting, the switching functions are constant and given by $0$ and $\pi$. Consequently,
$D_{\x}\theta_{j}(\x)\equiv 0$, and the integrals in~\eqref{prom} provide
\begin{align*}
\Delta_{1}(r,z;\alpha)
&=
\left(
\frac{\pi\alpha r}{2}
+\frac{1}{2}\big(-\pi r-4z+\pi^{2}+4\big),
\;
2r+\frac{\pi}{2}(z-5)
\right),\\[2mm]
\Delta_{2}(r,z;\alpha)
&=
\Bigg(
\frac{1}{8}\Big(
\pi\big(\alpha(\pi(\alpha-2)+8)-4\big)r
-8z\big((\pi-2)\alpha+2b\big)
+2\pi(\pi^{2}+4)\alpha
\Big)
\\
&\qquad
+\frac{\pi}{8}
\Big(
\pi\Big(\frac{32b(3z-5)}{8+9\pi^{2}}+r-\pi\Big)+16
\Big),
\\
&\qquad
\pi\left(\frac{1}{2}(\alpha-2)r+z+\pi\right)
+\frac{\pi^{2}}{8}(z-5)\left(1-\frac{32bz}{(9\pi^{2}+8)r}\right)
+8r-4z+4
\Bigg).
\end{align*}

Now, in order to check hypotheses \textbf{H1} and \textbf{H2}, define
\[
r_{0}(\alpha)=\frac{\pi\,(16-\pi^{2})}{\pi^{2}(\alpha-1)+16},
\qquad
z_{0}(\alpha)=\frac{\pi^{2}(5\alpha-1)+16}{\pi^{2}(\alpha-1)+16}.
\]
Observe that 
\[
(r_{0}(\alpha),z_{0}(\alpha))\in\overline{D}\quad \text{and}\quad\Delta_{1}\bigl(r_{0}(\alpha),z_{0}(\alpha);\alpha\bigr)=0,
\qquad \forall\,\alpha\in\overline{I}.
\]
Moreover, the eigenvalues of the Jacobian matrix
$D_{(r,z)}\Delta_{1}\bigl(r_{0}(\alpha),z_{0}(\alpha);\alpha\bigr)$
are given by $a(\alpha)\pm i\,b(\alpha)$, where
\[
a(\alpha)=\frac{\pi\alpha}{4},
\qquad
b(\alpha)=\frac{\sqrt{64-\pi^{2}(\alpha-2)^{2}}}{4}.
\]
Consequently, hypotheses \textbf{H1} and \textbf{H2} are satisfied at
$\alpha_{0}=0$.

Also, since
\begin{equation*}
\det\!\left(D_{(r,z)}\Delta_{1}\bigl(r_{0}(\alpha),z_{0}(\alpha);\alpha\bigr)\right)
=\tfrac{1}{4}\pi^{2}(\alpha-1)+4\neq 0, \quad \forall \,\alpha\in\overline{I},
\end{equation*}
the Implicit Function Theorem ensures the existence of
$\varepsilon_{1}>0$ and a unique smooth function $
\sigma:\overline{I}\times(-\varepsilon_{1},\varepsilon_{1})\to\overline{D}$ 
such that
\[
\sigma(\alpha,0)=\bigl(r_{0}(\alpha),z_{0}(\alpha)\bigr)
\quad\text{and}\quad
\mathcal{P}_{T}\bigl(\sigma(\alpha,\varepsilon);\alpha,\varepsilon\bigr)
=\sigma(\alpha,\varepsilon), \quad \forall\, (\alpha,\varepsilon)\in\overline{I}\times(-\varepsilon_{1},\varepsilon_{1})
\]
In particular, $\sigma(\alpha,\varepsilon)$ is a fixed point of
$\mathcal{P}_{T}$ and therefore corresponds to an isolated periodic solution
$\Psi(t,\alpha,\varepsilon)$ of \eqref{ds2}.

Furthermore, $\sigma(\alpha,\varepsilon)$ admits the expansion
\[
\sigma(\alpha,\varepsilon)
=\bigl(r_{0}(\alpha),z_{0}(\alpha)\bigr)
+\varepsilon\bigl(R_{1}(\alpha),S_{1}(\alpha)\bigr)
+\mathcal{O}(\varepsilon^{2}),
\]
where
\begin{equation*}
\begin{aligned}
R_{1}(\alpha)
&=\frac{\pi\, r_{1}(\alpha)}
{4\bigl(8+9\pi^{2}\bigr)\bigl(\pi^{2}(\alpha-1)+16\bigr)^{2}},\\[2mm]
S_{1}(\alpha)
&=\frac{z_{1}(\alpha)}
{\bigl(8+9\pi^{2}\bigr)\bigl(\pi^{2}(\alpha-1)+16\bigr)^{2}},
\end{aligned}
\end{equation*}
and
\begin{align*}
r_{1}(\alpha)=&\;8 \pi^{5} (17 \alpha-35) \alpha
-16 \pi^{4}\bigl(\alpha (45 \alpha-25 b+149)+13 b-196\bigr) \\
&-128 \pi^{2}\bigl(-29 b+302+5 \alpha (\alpha+3 b+7)\bigr)
-2048\bigl(\alpha+3 (b+6)\bigr)\\
&-9 \pi^{7} (\alpha-1) \alpha
+128 \pi^{3} (\alpha+16)\alpha
+2048 \pi  \alpha
+36 \pi^{6} (2 \alpha-1),
\\[2mm]
z_{1}(\alpha)=&\;128 \pi^{2}\bigl(\alpha (5 \alpha-9 b+7)-21 b+14\bigr)
-16 \pi^{4}\bigl(\alpha (10 \alpha (b-4)+13 b+89)-11 b-119\bigr)\\
&+2048 (\alpha-b)
-9 \pi^{7} (\alpha-1)
-18 \pi^{6}\bigl(5 (\alpha-2) \alpha+7\bigr)
+8 \pi^{5} (17 \alpha-35)\\
&+128 \pi^{3} (\alpha+16)
+2048 \pi .
\end{align*}

Finally, we compute the expansion in $\varepsilon$ of the first Lyapunov coefficient $\ell_{1}^{\varepsilon}$, as given by \eqref{l1forIdentity}.

In order to bring the map $\mathcal{P}_{T}(r,z;\alpha,\varepsilon)$ into the normal form \eqref{Identityperchange}, a sequence of changes of variables is required. First, the equilibrium point $\sigma(\alpha,\varepsilon)$ is translated to the origin by setting
\[
(r,z)=(u,v)+\sigma(\alpha,\varepsilon).
\]
Next, we compute the parameter $\beta(\varepsilon)$ as indicated in Lemma~\ref{lem2.2} and shift the parameter by writing $\alpha=\tau+\beta(\varepsilon)$, where
\[
\beta(\varepsilon)
=
-\varepsilon\left(\frac{24 b}{8+9\pi^{2}}-1\right)
+\mathcal{O}(\varepsilon^{2}).
\]
For the remaining transformations, we assume that we are at the critical parameter value $\tau=0$, that is, $\alpha=\beta(\varepsilon)$. In order for the resulting system to satisfy hypothesis $\textbf{H3}$, we introduce the linear transformation defined by the matrix
\begin{equation}\label{Linearchange}
	\begin{aligned}
		\mathcal{L}(\varepsilon)
		&=
		\begin{pmatrix}
			-\dfrac{\pi}{2} & -\dfrac{1}{2}\sqrt{16-\pi^{2}} \\
			2 & 0
		\end{pmatrix}\\
		&+\varepsilon
		\begin{pmatrix}
			-\dfrac{16\pi b}{8+9\pi^{2}}-\dfrac{\pi^{2}}{4}+\pi+2
			&
			\dfrac{\sqrt{16-\pi^{2}}\bigl(-16 b-(\pi-4)(8+9\pi^{2})\bigr)}{4(8+9\pi^{2})}
			\\[2mm]
			0 & 0
		\end{pmatrix}
		+\mathcal{O}(\varepsilon^{2}).
	\end{aligned}
\end{equation}
We then apply the transformation
\[
(u,v)=\mathcal{L}(\varepsilon)\,\mathbf{y}^{\intercal},
\qquad
\mathbf{y}=(y_{1},y_{2}).
\]
With these transformations, the map attains the desired expansion
\[
\boldsymbol{\mathcal{H}}_{\varepsilon}(\mathbf{y},0)
=
\mathcal{A}_{\varepsilon}\mathbf{y}
+\frac{1}{2}\mathcal{B}_{\varepsilon}(\mathbf{y},\mathbf{y})
+\frac{1}{6}\mathcal{C}_{\varepsilon}(\mathbf{y},\mathbf{y},\mathbf{y})
+\mathcal{O}(||\mathbf{y}||^{4}),
\]
where
\begin{align*}
\mathcal{A}_{\varepsilon}
&=
\begin{pmatrix}
1+\left(\dfrac{\pi^{2}}{8}-2\right)\varepsilon^{2}
&
-\dfrac{1}{2}\sqrt{16-\pi^{2}}\,\varepsilon-\sqrt{16-\pi^{2}}\,\varepsilon^{2}
\\[2mm]
\dfrac{1}{2}\sqrt{16-\pi^{2}}\,\varepsilon+\sqrt{16-\pi^{2}}\,\varepsilon^{2}
&
1+\left(\dfrac{\pi^{2}}{8}-2\right)\varepsilon^{2}
\end{pmatrix}
+\mathcal{O}(\varepsilon^{3}).
\end{align*}
Also, denoting $\mathbf{u}=(u_{1},u_{2})$, $\mathbf{v}=(v_{1},v_{2})$, and $\mathbf{w}=(w_{1},w_{2})$, the bilinear form $\mathcal{B}_{\varepsilon}$ can be written as
\[
\mathcal{B}_{\varepsilon}(\mathbf{u},\mathbf{v})
=
\varepsilon^{2}\bigl(\mathcal{B}^{1}_{\varepsilon}(\mathbf{u},\mathbf{v}),
\mathcal{B}^{2}_{\varepsilon}(\mathbf{u},\mathbf{v})\bigr)
+\mathcal{O}(\varepsilon^{3}),
\]
where
\begin{align*}
\mathcal{B}^{1}_{\varepsilon}(\mathbf{u},\mathbf{v})
&=
\dfrac{b\Bigl(5\pi\sqrt{16-\pi^{2}}(u_{1}v_{2}+u_{2}v_{1})
-2\pi^{2}u_{2}v_{2}+32u_{2}v_{2}\Bigr)}
{\pi(8+9\pi^{2})},\\[2mm]
\mathcal{B}^{2}_{\varepsilon}(\mathbf{u},\mathbf{v})
&=
-\dfrac{b\Bigl(5\pi\sqrt{16-\pi^{2}}(u_{1}v_{2}+u_{2}v_{1})
-2\pi^{2}u_{2}v_{2}+32u_{2}v_{2}\Bigr)}
{\sqrt{16-\pi^{2}}(8+9\pi^{2})},
\end{align*}
and the trilinear form $\mathcal{C}_{\varepsilon}$ admits the expansion
\[
\mathcal{C}_{\varepsilon}(\mathbf{u},\mathbf{v},\mathbf{w})
=
\varepsilon^{2}\bigl(\mathcal{C}^{1}_{\varepsilon}(\mathbf{u},\mathbf{v},\mathbf{w}),
\mathcal{C}^{2}_{\varepsilon}(\mathbf{u},\mathbf{v},\mathbf{w})\bigr)
+\mathcal{O}(\varepsilon^{3}),
\]
where
\begin{align*}
\mathcal{C}^{1}_{\varepsilon}(\mathbf{u},\mathbf{v},\mathbf{w})
&=
w_{1}\left(
\dfrac{5\sqrt{16-\pi^{2}}\,b\,(v_{1}v_{2}+u_{2}v_{1})}{3(8+9\pi^{2})}
-\dfrac{2(\pi^{2}-16)b\,u_{2}v_{2}}{\pi(8+9\pi^{2})}
\right)\\
&\quad
-w_{2}\left(
\dfrac{\sqrt{16-\pi^{2}}\,b\bigl(-5\pi^{2}v_{1}^{2}
+3\pi^{2}u_{2}v_{2}-48u_{2}v_{2}\bigr)}{3\pi^{2}(8+9\pi^{2})}
+\dfrac{2(\pi^{2}-16)b\,(v_{1}v_{2}+u_{2}v_{1})}{\pi(8+9\pi^{2})}
\right),
\end{align*}
and
\begin{align*}
\mathcal{C}^{2}_{\varepsilon}(\mathbf{u},\mathbf{v},\mathbf{w})
&=
w_{1}\left(
-\dfrac{5\pi b\,(v_{1}v_{2}+u_{2}v_{1})}{3(8+9\pi^{2})}
-\dfrac{2\sqrt{16-\pi^{2}}\,b\,u_{2}v_{2}}{8+9\pi^{2}}
\right)\\
&\quad
-w_{2}\left(
\dfrac{b\bigl(5\pi^{2}v_{1}^{2}
-3\pi^{2}u_{2}v_{2}+48u_{2}v_{2}\bigr)}{3\pi(8+9\pi^{2})}
+\dfrac{2\sqrt{16-\pi^{2}}\,b\,(v_{1}v_{2}+u_{2}v_{1})}{8+9\pi^{2}}
\right).
\end{align*}

Before computing the Lyapunov coefficients as described in Section~\ref{lyapcoefficients}, we observe that, in the present setting, the functions $B_{1}(\bu,\bv)$ and $C_{1}(\bu,\bv,\w)$ appearing in the expansions
\eqref{multilinearfunctions}--\eqref{multilinearfunctions2} vanish. Consequently, the multilinear terms take the form
\begin{align*}
\mathcal{B}_{\varepsilon}(\bu,\bv) &= \varepsilon^{2} B_{2}(\bu,\bv) + \mathcal{O}(\varepsilon^{3}),\\
\mathcal{C}_{\varepsilon}(\bu,\bv,\w) &= \varepsilon^{2} C_{2}(\bu,\bv,\w) + \mathcal{O}(\varepsilon^{3}).
\end{align*}

As a direct consequence of equation \eqref{cofi} in the Appendix, we obtain
\[
g_{20}^{1} = g_{11}^{1} = g_{02}^{1} = g_{21}^{1} = 0.
\]
Proceeding with the computation of the remaining coefficients, we find
\begin{align*}
\textrm{Re}(g_{20}) &= \varepsilon^{2}\textrm{Re}(g_{20}^{2}) + \mathcal{O}(\varepsilon^{3}) \\
&= \varepsilon^{2}\Bigg(
-\frac{8 \sqrt{\frac{2}{16-\pi^{2}}}\, b}{8+9\pi^{2}}
-\frac{5\pi b}{\sqrt{2}\,(8+9\pi^{2})}
\Bigg)
+ \mathcal{O}(\varepsilon^{3}),\\[0.3cm]
\textrm{Re}(g_{11}) &= \varepsilon^{2}\textrm{Re}(g_{11}^{2}) + \mathcal{O}(\varepsilon^{3}) \\
&= \varepsilon^{2}\Bigg(
\frac{8 \sqrt{\frac{2}{16-\pi^{2}}}\, b}{8+9\pi^{2}}
\Bigg)
+ \mathcal{O}(\varepsilon^{3}),\\[0.3cm]
\textrm{Re}(g_{02}) &= \varepsilon^{2}\textrm{Re}(g_{02}^{2}) + \mathcal{O}(\varepsilon^{3}) \\
&= \varepsilon^{2}\Bigg(
\frac{8 \sqrt{\frac{2}{16-\pi^{2}}}\, b}{8+9\pi^{2}}
+\frac{5\pi b}{\sqrt{2}\,(8+9\pi^{2})}
\Bigg)
+ \mathcal{O}(\varepsilon^{3}),\\[0.3cm]
\textrm{Re}(g_{21}) &= \varepsilon^{2}\textrm{Re}(g_{21}^{2}) + \mathcal{O}(\varepsilon^{3}) \\
&= \varepsilon^{2}\Bigg(
\frac{4 b}{\pi(8+9\pi^{2})}
-\frac{2\pi b}{3(8+9\pi^{2})}
\Bigg)
+ \mathcal{O}(\varepsilon^{3}).
\end{align*}

Finally, substituting these expressions into formula \eqref{l1PSDE} in the Appendix, we obtain the Taylor expansion of the first Lyapunov coefficient:
\begin{equation}\label{applilyco}
\ell_{1}^{\varepsilon}
= \varepsilon^{2}\textrm{Re}(g_{21}^{2}) + \mathcal{O}(\varepsilon^{3}) = -\varepsilon^{2}\frac{(\pi^{2}-6)b}{3\pi(8+9\pi^{2})}
+ \mathcal{O}(\varepsilon^{3}).
\end{equation}
Therefore, it follows from \eqref{applilyco} that
\[
\ell_{1,1}=0\quad\text{and}\quad \ell_{1,2}
= -\frac{(\pi^{2}-6)b}{3\pi(8+9\pi^{2})}.
\]
The conclusion then follows directly from Theorem~\ref{MainTheorem}.
\end{proof}

\section{Conclusions}
In this work, we investigated the emergence of invariant tori in non-smooth dynamical systems, with particular emphasis on piecewise linear differential equations. Motivated by the relevance of such models in applications and by the scarcity of analytical results concerning higher-dimensional invariant objects in this setting, we developed a framework that combines perturbative methods with bifurcation analysis of near-identity maps.

Our first main contribution was to establish sufficient conditions for the existence and stability of limit tori in a broad class of non-autonomous $T$-periodic piecewise smooth differential systems (Theorem~\ref{MainTheorem}). By proving that the associated time-$T$ map is smooth and admits a near-identity expansion, we reduced the problem to the analysis of Neimark--Sacker bifurcation in near-identity maps (Theorem~\ref{TheoremA}). This approach extends recent developments in the smooth setting to the non-smooth framework, thereby providing an analytical mechanism for torus bifurcation in piecewise smooth systems.

As a second and more concrete contribution (Theorem~\ref{mainresult}), we applied the general theory to a family of 3D piecewise linear differential systems. We proved analytically the existence of a limit torus arising from a Neimark--Sacker bifurcation of the associated time-$T$ map. As far as we are aware, this constitutes the first example of a 3D piecewise linear system exhibiting an isolated invariant torus, analytic detected. Moreover, we characterized its stability and showed that the torus persists under sufficiently small perturbations.

Beyond the specific results obtained here, this work opens several directions for future research. Possible extensions include the development of a theory of normal hyperbolicity in the piecewise smooth context and, subsequently, the extension of higher-order averaging techniques to the investigation of invariant tori in higher-dimensional piecewise smooth systems, thereby generalizing the results of~\cite{novaes2024invariant,pereira2023mechanism}. We believe that the methods introduced in this paper provide a solid foundation for advancing the understanding of higher-dimensional invariant structures in non-smooth dynamical systems.

\section*{Acknowledgements}
MRC was partially supported by the São Paulo Research Foundation (FAPESP), Grant No. 2023/06076-0. DDN was partially supported by the São Paulo Research Foundation (FAPESP), Grant No. 2024/15612-6; by the Conselho Nacional de Desenvolvimento Científico e Tecnológico (CNPq), Grant No. 301878/2025-0; and by the Coordenação de Aperfeiçoamento de Pessoal de Nível Superior - Brasil (CAPES), through the MATH-AmSud program, Grant No. 88881.179491/2025-01. JSGR was partially supported by the Coordenação de Aperfeiçoamento de Pessoal de Nível Superior - Brasil (CAPES) - Finance Code 001.

\section{Appendix: Lyapunov Coefficients}\label{lyapcoefficients}

Now, considering the Taylor expansion up to order $2$, we explain how we can get an expression of $\ell_{1}^{\e}$ when $r=1$ and $k=2$. To begin, for $i=1,2$, we define
\begin{equation}\label{notaBiandCi}
\begin{aligned}
    \widetilde{B}_{1}^{i}(\bu,\bv)&=\frac{d}{d\e}\, \B_{\e}^{i}(\bu,\bv)\bigg|_{\e=0}, & \qquad\widetilde{C}_{1}^{i}(\mathbf q,\mathbf q, \bar{\mathbf q})&=\frac{d}{d\e}\,  \mathcal{C}_{\e}^{i}(\mathbf q,\mathbf q, \bar{\mathbf q})\bigg|_{\e=0},\\
    \widetilde{B}_{2}^{i}(\bu,\bv)&=\frac{d}{d\e^{2}}\, \B_{\e}^{i}(\bu,\bv)\bigg|_{\e=0}, &  \qquad \widetilde{C}_{2}^{i}(\mathbf q,\mathbf q, \bar{\mathbf q})&=\frac{d}{d\e^{2}}\, \mathcal{C}_{\e}^{i}(\mathbf q,\mathbf q, \bar{\mathbf q})\bigg|_{\e=0},
\end{aligned}    
\end{equation}
  and denote
  \begin{align*}
  \widetilde{B}_{j}(\bu,\bv)&=(\widetilde{B}_{j}^{1}(\bu,\bv), \widetilde{B}_{j}^{2}(\bu,\bv)),\\   
  \widetilde{C}_{j}(\mathbf q,\mathbf q,\bar{\mathbf q})&=(\widetilde{C}_{j}^{1}(\mathbf q,\mathbf q, \bar{\mathbf q}),\widetilde{C}_{j}^{2}(\mathbf q,\mathbf q, \bar{\mathbf q})),
  \end{align*}
for $j=1,2$. 

Moreover
\begin{equation}\label{cofi}
\begin{aligned}
 g_{11}^{j}&=\left<\mathbf q,\widetilde{B}_{j}(\mathbf q, \bar{\mathbf q})\right>,\,\,
    g_{20}^{j}=\left<\mathbf q,\widetilde{B}_{j}(\mathbf q, \mathbf q)\right>,\\   
    g_{02}^{j}&=\left<\mathbf q,\widetilde{B}_{j}(\bar{\mathbf q}, \bar{\mathbf q})\right>,\,\, g_{21}^{j}=\left<\mathbf q,\widetilde{C}_{j}(\mathbf q,\mathbf q, \bar{\mathbf q})\right>,
\end{aligned}
      \end{equation}
for $j=1,2$. 

Using equation \eqref{expressionsg_ij's}, we get
\begin{align*}
g_{11}&=  \left<\mathbf q,\e\widetilde{B}_{1}(\mathbf q, \bar{\mathbf q})\right> + \left<\mathbf q,\e^{2}\widetilde{B}_{2}(\mathbf q, \bar{\mathbf q})\right> +\mathcal{O}(\e^{3})\\&=\e\,g_{11}^{1} + \e^{2}\,g_{11}^{2} + \mathcal{O}(\e^{3}).  
\end{align*}

Analogously, we obtain 
\begin{align*}
g_{20}&=\e\,g_{20}^{1} + \e^{2}\,g_{20}^{2} + \mathcal{O}(\e^{3}),\\
g_{02}&=\e\,g_{02}^{1} + \e^{2}\,g_{02}^{2} + \mathcal{O}(\e^{3}),\\
g_{21}&=\e\,g_{21}^{1} + \e^{2}\,g_{21}^{2} + \mathcal{O}(\e^{3}).
\end{align*}

On the other hand, using equation \eqref{expreeigenIdentity}, we have
\begin{equation}\label{expreeigen}
    e^{i\theta_{\e}}=1 + \sum_{j=1}^{2}\e^{j}(a_{j}+i\, b_{j}) +\mathcal{O}(\e^{3}),
\end{equation}

However, from Hypothesis \textbf{{H1}}, we have that $a_{1}=0$ and $b_{1}=b_{0}>0$, so $\widetilde{a}_{\e}=\e^{2}\,a_{2}$, and $\widetilde{b}_{\e}=\e\,b_{0}+\e^{2}\,b_{2}$. In consequence, the equation \eqref{expreeigen} writes 
$$e^{i\theta_{\e}}=\kappa_{\e}(0)=1 + i\e\,b_{0}+\e^{2}(a_{2}+i\,b_{2}) +\mathcal{O}(\e^{3}).$$ 

Thus, we get the following expressions
\begin{equation}\label{gij}
\begin{aligned}
  \frac{e^{-i\theta_{\e}}g_{21}}{2}=&\e\frac{g_{21}^{1}}{2} + \frac{\e^{2}}{2}\left(g_{21}^{2}-ib_{0}g_{21}^{1}\right)+\mathcal{O}(\e^{3}),\\
\frac{(1-2\,e^{i\theta_{\e}})\,e^{-2i\theta_{\e}}}{2(1-e^{i\theta_{\e}})}\,g_{20}g_{11}=&-\e\frac{i}{2b_{0}}\,g_{11}^{1}\,g_{20}^{1}-\e^{2}\frac{i}{2{b_{0}}^{2}}\big(b_{0}\,g_{11}^{2}g_{20}^{1} + b_{0}\,g_{11}^{1}g_{20}^{2}\\
&-b_{2}\,g_{11}^{1}g_{20}^{1}+ ia_{2}\,g_{11}^{1}g_{20}^{1}\big)+\mathcal{O}(\e^{3}),\\
|g_{11}|^{2}=& \frac{1}{2}\left(\e^{2}\,|g_{11}^{1}|^{2}+\mathcal{O}(\e^{3})\right),\\
 |g_{02}|^{2}=& \frac{1}{2}\left(\e^{2}\,|g_{02}^{1}|^{2}+\mathcal{O}(\e^{3})\right).
\end{aligned}    
\end{equation}

Therefore, substituting all the expressions in equation \eqref{gij} into equation \eqref{lyapuforIdentity}, we obtain
\begin{equation}\label{l1PSDE}
\begin{aligned}
\ell_{1}^{\e}=&\frac{\e}{2}\bigg(\text{Re}\,(g_{21}^{1}) + \frac{1}{b_{0}}\text{Re}\,(i\,g_{11}^{1}g_{20}^{1})\bigg) + \frac{\e^{2}}{2}\bigg(\text{Re}\,(g_{21}^{2})-b_{0}\text{Re}\,(i\,g_{21}^{1})+\frac{1}{b_{0}}\text{Re}(i\,g_{11}^{2}g_{20}^{1})\\
   &+\frac{1}{b_{0}}\text{Re}\,(i\,g_{11}^{1}g_{20}^{2})-\frac{b_{2}}{{b_{0}}^{2}}\text{Re}\,(i\,g_{11}^{1}g_{20}^{1})-\frac{a_{2}}{{b_{0}}^{2}}\text{Re}\,(g_{11}^{1}g_{20}^{1}) - |g_{11}^{1}|^{2}-\frac{1}{2}|g_{02}^{1}|^{2}\bigg)+\mathcal{O}(\e^{3})\cdot    
\end{aligned}    
\end{equation}

Moreover,
\begin{align*}
 \ell_{1,1}= \frac{1}{2}\bigg(\text{Re}\,(g_{21}^{1}) &+ \frac{1}{b_{0}}\text{Re}\,(i\,g_{11}^{1}g_{20}^{1})\bigg),\\
 \ell_{1,2}=\frac{1}{2}\bigg(\text{Re}\,(g_{21}^{2})&-b_{0}\text{Re}\,(i\,g_{21}^{1})+\frac{1}{b_{0}}\text{Re}(i\,g_{11}^{2}g_{20}^{1})
   +\frac{1}{b_{0}}\text{Re}\,(i\,g_{11}^{1}g_{20}^{2})\\
   &-\frac{b_{2}}{{b_{0}}^{2}}\text{Re}\,(i\,g_{11}^{1}g_{20}^{1})-\frac{a_{2}}{{b_{0}}^{2}}\text{Re}\,(g_{11}^{1}g_{20}^{1}) - |g_{11}^{1}|^{2}-\frac{1}{2}|g_{02}^{1}|^{2}\bigg)\cdot
\end{align*}

Thus, for instance, using the above expression for $\ell_{1,1}$ together with equations~\eqref{notaBiandCi} and~\eqref{cofi}, we obtain the first-order term in $\varepsilon$ of the Lyapunov coefficient expansion as follows:
\begin{equation*}\label{l11forIdentity}
 \begin{aligned}
 \ell_{1,1}&=\frac{1}{8}\bigg(\parde[^{3}\f_{1}^{2}(\x_{\alpha_{0}};\alpha_{0})]{y^{3}} + \parde[^{3}\f_{1}^{1}(\x_{\alpha_{0}};\alpha_{0})]{x\partial y^{2}} + \parde[^{3}\f_{1}^{2}(\x_{\alpha_{0}};\alpha_{0})]{x^{2}\partial y} + \parde[^{3}\f_{1}^{1}(\x_{\alpha_{0}};\alpha_{0})]{x^{3}}\bigg)\\ 
 &+ \frac{1}{8\,b_{0}}\bigg(\parde[^{2}\f_{1}^{1}(\x_{\alpha_{0}};\alpha_{0})]{y^{2}}\parde[^{2}\f_{1}^{2}(\x_{\alpha_{0}};\alpha_{0})]{y^{2}} + \parde[^{2}\f_{1}^{1}(\x_{\alpha_{0}};\alpha_{0})]{y^{2}}\parde[^{2}\f_{1}^{1}(\x_{\alpha_{0}};\alpha_{0})]{x\partial y}\\
 &-\parde[^{2}\f_{1}^{2}(\x_{\alpha_{0}};\alpha_{0})]{y^{2}}\parde[^{2}\f_{1}^{2}(\x_{\alpha_{0}};\alpha_{0})]{x\partial y} + \parde[^{2}\f_{1}^{1}(\x_{\alpha_{0}};\alpha_{0})]{x\partial y}\parde[^{2}\f_{1}^{1}(\x_{\alpha_{0}};\alpha_{0})]{x^{2}}\\
 &- \parde[^{2}\f_{1}^{2}(\x_{\alpha_{0}};\alpha_{0})]{x\partial y}\parde[^{2}\f_{1}^{2}(\x_{\alpha_{0}};\alpha_{0})]{x^{2}} - \parde[^{2}\f_{1}^{1}(\x_{\alpha_{0}};\alpha_{0})]{x^{2}}\parde[^{2}\f_{1}^{2}(\x_{\alpha_{0}};\alpha_{0})]{x^{2}}\bigg)\cdot
\end{aligned}   
\end{equation*}

\bibliographystyle{abbrv}
\bibliography{Bibliography}
\end{document}